\documentclass[11pt, oneside]{amsart}       
\usepackage{geometry}                       
\usepackage[parfill]{parskip}           
\usepackage{graphicx}
\usepackage{color}
\usepackage{amssymb,amsmath}
\newtheorem{theorem}{Theorem}
\newtheorem{lemma}[theorem]{Lemma}

\theoremstyle{definition}
\newtheorem{remark}[theorem]{Remark}
\newtheorem{example}[theorem]{Example}
\newtheorem{problem}[theorem]{Problem}
\newtheorem{definition}[theorem]{Definition}
\usepackage{epstopdf}
\numberwithin{equation}{section}
\numberwithin{theorem}{section}

\author{P.D.~Dragnev}
\address{Department of Mathematical Sciences, Indiana University--Purdue University Fort Wayne, Fort Wayne, IN 46805,
          USA}
          \email{dragnevp@ipfw.edu}

\author{B.~Fuglede}
\address{Department of Mathematical Sciences, University of Copenhagen, Universitetsparken 5,
2100 Copenhagen, Denmark}
\email{fuglede@math.ku.dk}

\author{D.P.~Hardin}
\address{Center for Constructive Approximation, Department of Mathematics,
Vanderbilt University,
Nashville, TN 37240, USA}
\email{doug.hardin@vanderbilt.edu}

\author{E.B.~Saff}
\address{Center for Constructive Approximation, Department of Mathematics,
Vanderbilt University,
Nashville, TN 37240, USA}
\email{edward.b.saff@vanderbilt.edu}

\author{N.~Zorii}
\address{Institute of Mathematics of
National Academy of Sciences of Ukraine, Tereshchenkivska 3, 01601,
Kyiv-4, Ukraine}
\email{natalia.zorii@gmail.com}

\thanks{The research of the first author was supported, in part, by a Simons Foundation grant no.\ 282207.
The research
of the third and the fourth authors was supported, in part, by the U.\ S.\ National Science Foundation under
grants DMS-1516400. The research of the fifth author was supported, in part, by the Scholar-in-Residence
program at IPFW and by the Department of Mathematical Sciences of the University of Copenhagen.}

\begin{document}

\title[Condensers with touching plates]{Condensers with touching plates and constrained minimum Riesz and
Green energy problems}

\begin{abstract}
We study minimum energy problems relative to the $\alpha$-Riesz
kernel $|x-y|^{\alpha-n}$, $\alpha\in(0,2]$, over signed Radon
measures $\mu$ on $\mathbb R^n$, $n\geqslant3$, associated with a
generalized condenser $(A_1,A_2)$, where $A_1$ is a relatively
closed subset of a domain $D$ and $A_2=\mathbb R^n\setminus D$. We
show that, though $A_2\cap\mathrm{Cl}_{\mathbb R^n}A_1$ may have
nonzero capacity, this minimum energy problem is uniquely solvable
(even in the presence of an external field) if we restrict ourselves
to $\mu$ with $\mu^+\leqslant\xi$, where a constraint $\xi$ is
properly chosen. We establish the sharpness of the sufficient
conditions on the solvability thus obtained, provide descriptions of
the weighted $\alpha$-Riesz potentials of the solutions, single out
their characteristic properties, and analyze their supports. The
approach developed is mainly based on the establishment of an
intimate relationship between the constrained minimum $\alpha$-Riesz
energy problem over signed measures associated with $(A_1,A_2)$ and
the constrained minimum $\alpha$-Green energy problem over positive
measures carried by $A_1$. The results are illustrated by examples.
\end{abstract}

\renewcommand{\thefootnote}{}

\footnote{2010 \emph{Mathematics Subject Classification}: 31C15.}

\footnote{\emph{Key words and phrases}: Constrained minimum energy problems, $\alpha$-Riesz kernels,
$\alpha$-Green kernels, external fields,
condensers with touching plates.}

\renewcommand{\thefootnote}{\arabic{footnote}}
\setcounter{footnote}{0}

\maketitle

\section{Introduction}

The purpose of this paper is to study minimum energy problems with
external fields (also known in the literature as {\it weighted
minimum energy problems\/}) relative to the $\alpha$-Riesz kernel
$\kappa_\alpha(x,y):=|x-y|^{\alpha-n}$ of order $\alpha\in(0,2]$ on
$\mathbb R^n$, $n\geqslant3$, where $|x-y|$ is the Euclidean
distance between $x,y\in\mathbb R^n$ and infimum is taken over
classes of (signed) Radon measures $\mu$ on $\mathbb R^n$ associated
with a generalized condenser $\mathbf A=(A_1,A_2)$. More precisely,
an ordered pair $\mathbf A=(A_1,A_2)$ is termed a {\it generalized
condenser\/} in $\mathbb R^n$ if $A_1$ is a relatively closed subset
of a given (connected open) domain $D\subset\mathbb R^n$ and
$A_2=D^c:=\mathbb R^n\setminus D$, while $\mu$ is said to be {\it
associated with\/} $\mathbf A$ if the positive and negative parts in
the Hahn--Jor\-dan decomposition of $\mu$ are carried by $A_1$ and
$A_2$, respectively.

Note that, although $A_1\cap A_2=\varnothing$, the set
$A_2\cap\mathrm{Cl}_{\mathbb R^n}A_1$ may have nonzero (in
particular infinite, see Example~\ref{ex-3} below) $\alpha$-Riesz
capacity and may even coincide with the boundary of $D$ relative to
$\mathbb R^n$. Therefore the classical {\it condenser problem\/} for
the generalized condenser $\mathbf A$, which amounts to the minimum
$\alpha$-Riesz energy problem over the class of all $\mu$ associated
with $\mathbf A$ and normalized by $\mu^+(A_1)=\mu^-(A_2)=1$, can
easily be shown to have {\it no\/} solution, see
Theorem~\ref{pr1uns}. Using the electrostatic interpretation, which
is possible for the Coulomb kernel $|x-y|^{-1}$ on $\mathbb R^3$, in
the case where a minimum energy problem has no solution we say that
a short-circuit occurs between the oppositely charged plates of the
generalized condenser $\mathbf A$. It is therefore meaningful to ask
what kinds of additional requirements on the objects in question
will prevent this blow-up effect, and secure that a solution to the
corresponding minimum $\alpha$-Riesz energy problem does exist.

We show that a solution $\lambda_{\mathbf A}^\xi$ to the minimum
$\alpha$-Riesz energy problem exists (no short-circ\-uit occurs) if
we restrict ourselves to $\mu$ with $\mu^+\leqslant\xi$, where the
constraint $\xi$ is properly chosen. More precisely, if $A_2=D^c$ is
not $\alpha$-thin at infinity, then such $\lambda_{\mathbf A}^\xi$
exists (even in the presence of an external field) provided that
$\xi$ is a positive Radon measure carried by $A_1$ with finite
$\alpha$-Riesz energy
$E_{\kappa_\alpha}(\xi):=\int\kappa_\alpha(x,y)\,d(\xi\otimes\xi)(x,y)<\infty$
and with total mass $\xi(A_1)\in(1,\infty)$; see
Theorem~\ref{th-suff}.\footnote{See Section~\ref{sec:RG} for the
notion of $\alpha$-thinness at infinity. The uniqueness of a
solution $\lambda_{\mathbf A}^\xi$ can be established by standard
methods based on the convexity of the class of admissible measures
and the pre-Hilbert structure on the linear space of all (signed)
Radon measures on $\mathbb R^n$ with
$E_{\kappa_\alpha}(\mu)<\infty$, see Lemma~\ref{l:uniq}.} In
particular, if the domain $D$ is bounded, then a solution
$\lambda_{\mathbf A}^\xi$ exists whenever $\mathbf A=(D,D^c)$,
$m_n(D)>1$ and $\xi=m_n|_D$, where $m_n$ is the $n$-dimen\-sional
Lebesgue measure on $\mathbb R^n$. Theorem~\ref{th-suff} is sharp in
the sense that it no longer holds if the requirement
$\xi(A_1)<\infty$ is omitted from its hypotheses, see
Theorem~\ref{th-unsuff}.

We provide descriptions of the weighted $\alpha$-Riesz potentials of the solutions $\lambda_{\mathbf A}^\xi$,
single out their characteristic properties, and analyze their supports, see Theorems~\ref{desc-pot},
\ref{desc-sup} and \ref{zone}. The results are illustrated by Examples~\ref{ex-2} and~\ref{ex-3}.
The theory of weighted minimum $\alpha$-Riesz energy problems with a (positive) constraint $\xi$ acting only
on positive parts of measures associated with $\mathbf A$, thus developed, remains valid in its full
generality for the signed constraint $\xi-\xi^{D^c}$ acting simultaneously on the positive and negative parts
of the measures in question, see Section~\ref{entire}. (Here $\xi^{D^c}$ is the $\alpha$-Riesz balayage
of $\xi$ onto~$D^c$.)

The approach developed is mainly based on the establishment of an
intimate relationship between, on the one hand, the constrained
weighted minimum $\alpha$-Riesz energy problem over (signed)
measures associated with $\mathbf A$ and, on the other hand, the
constrained weighted minimum $\alpha$-Green energy problem over
positive measures carried by $A_1$ (Theorem~\ref{th:rel}). The proof
of Theorem~\ref{th:rel} uses substantially the required finiteness
of $E_{\kappa_\alpha}(\xi)$. Regrettably, a similar assertion in
\cite{DFHSZ}, Lemma~4.2, did not require that
$E_{\kappa_\alpha}(\xi)<\infty$, being based on a false statement,
Lemma~2.4, that the finiteness of the $\alpha$-Green energy
$E_g(\mu)$ of a bounded measure $\mu$ on $D$ implies the finiteness
of its $\alpha$-Riesz energy (see Example~\ref{counterex} below for
a counterexample). This caused the incorrectness of some of the
formulations and proofs presented in \cite{DFHSZ}. The present paper
rectifies the results on the constrained weighted $\alpha$-Riesz and
$\alpha$-Green energy problems announced in~\cite{DFHSZ}.

Regarding the constrained weighted minimum $\alpha$-Green energy
problem over positive measures carried by $A_1$, crucial to the
arguments applied in the investigation thereof is the perfectness of
the $\alpha$-Green kernel $g$ on a domain $D$, established recently
by the second and the fifth named authors \cite{FZ}, which amounts
to the completeness of the cone of all positive Radon measures $\nu$
on $D$ with finite $\alpha$-Green energy $E_g(\nu)$ in the topology
determined by the energy norm $\|\nu\|_g:=\sqrt{E_g(\nu)}$.

\section{Preliminaries}\label{sec:princ}

Let $X$ be a locally compact (Hausdorff) space \cite[Chapter~I,
Section~9, n$^\circ$\,7]{B1}, to be specified below, and $\mathfrak
M(X)$ the linear space of all real-valued (signed) Radon measures
$\mu$ on $X$, equipped with the {\it vague\/} topology, i.e.\ the
topology of pointwise convergence on the class $C_0(X)$ of all
continuous functions on $X$ with compact support.\footnote{When
speaking of a continuous numerical function we understand that the
values are {\it finite\/} real numbers.} We refer to \cite{B2,E2}
for the theory of measures and integration on a locally compact
space, to be used throughout the paper; see also \cite{F1} for a
short survey.

For the purposes of the present study it is enough to assume that
$X$ is metrizable and {\it countable at infinity\/}, where the
latter means that $X$ can be represented as a countable union of
compact sets \cite[Chapter~I, Section~9, n$^\circ$\,9]{B1}. Then the
vague topology on $\mathfrak M(X)$ satisfies the first axiom of
countability \cite[Remark~2.5]{DFHSZ2}, and the vague convergence is
entirely determined by convergence of sequences. The vague topology
on $\mathfrak M(X)$ is Hausdorff, and hence a vague limit of any
sequence in $\mathfrak M(X)$ is unique (provided that it exists).

Let $\mu^+$ and $\mu^-$ denote the positive and negative parts of a
measure $\mu\in\mathfrak M(X)$ in the Hahn--Jor\-dan decomposition,
$|\mu|:=\mu^++\mu^-$ its total variation, and $S(\mu)=S^\mu_{X}$ its
support. A measure $\mu$ is said to be {\it bounded\/} if
$|\mu|(X)<\infty$. Given $\mu\in\mathfrak M(X)$ and a
$|\mu|$-meas\-ur\-able function $u:X\to[-\infty,\infty]$, we shall
for brevity write $\langle u,\mu\rangle:=\int
u\,d\mu$.\footnote{When introducing notation about numerical
quantities we always assume the corresponding object on the right to
be well defined (as a finite real number or~$\pm\infty$).}

Let $\mathfrak M^+(X)$ stand for the (convex, vaguely closed) cone of all positive $\mu\in\mathfrak M(X)$,
and let $\Psi(X)$ consist of all lower semicontinuous (l.s.c.)
functions $\psi:X\to(-\infty,\infty]$, nonnegative unless $X$ is compact. The following fact is well known,
see e.g.\ \cite[Section~1.1]{F1}.

\begin{lemma}\label{lemma-semi}For any\/ $\psi\in\Psi(X)$ the mapping\/ $\mu\mapsto\langle\psi,\mu\rangle$
is vaguely l.s.c.\ on\/~$\mathfrak M^+(X)$.\end{lemma}

We define a {\it kernel\/} $\kappa$ on $X$ as a symmetric positive
function from $\Psi(X\times X)$. Given $\mu,\nu\in\mathfrak M(X)$,
let $E_\kappa(\mu,\nu)$ and $U_\kappa^\mu$ denote the {\it mutual
energy\/} and the {\it potential\/} relative to the kernel $\kappa$,
i.e.
\begin{align*}
E_\kappa(\mu,\nu)&:=\int\kappa(x,y)\,d(\mu\otimes\nu)(x,y),\\
U_\kappa^\mu(\cdot)&:=\int\kappa(\cdot,y)\,d\mu(y).
\end{align*}

Observe that $U_\kappa^\mu(x)$, $\mu\in\mathfrak M(X)$, is well defined at $x\in X$ provided that
$U_\kappa^{\mu^+}(x)$ or $U_\kappa^{\mu^-}(x)$ is finite, and then
$U_\kappa^\mu(x)=U_\kappa^{\mu^+}(x)-U_\kappa^{\mu^-}(x)$. In particular, if $\mu\geqslant0$, then
$U_\kappa^\mu$ is defined everywhere on $X$ and represents a positive l.s.c.\ function,
see Lemma~\ref{lemma-semi}.

Also note that $E_\kappa(\mu,\nu)$, $\mu,\nu\in\mathfrak M(X)$, is
well defined and equal to $E_\kappa(\nu,\mu)$ provided that
$E_\kappa(\mu^+,\nu^+)+E_\kappa(\mu^-,\nu^-)$ or
$E_\kappa(\mu^+,\nu^-)+E_\kappa(\mu^-,\nu^+)$ is finite. For
$\mu=\nu$ the mutual energy $E_\kappa(\mu,\nu)$ becomes the {\it
energy\/} $E_\kappa(\mu):=E_\kappa(\mu,\mu)$ of $\mu$. Let $\mathcal
E_\kappa(X)$ consist of all $\mu\in\mathfrak M(X)$ whose energy
$E_\kappa(\mu)$ is finite, which by definition means that the kernel
$\kappa$ is $(|\mu|\otimes|\mu|)$-int\-egr\-able, i.e.\
$E_\kappa(|\mu|)<\infty$, and let $\mathcal E^+_\kappa(X):=\mathcal
E_\kappa(X)\cap\mathfrak M^+(X)$.

If $f:X\to[-\infty,\infty]$ is an {\it external field\/}, then the $f$-{\it weighted potential\/}
$W^\mu_{\kappa,f}$ and the $f$-{\it weight\-ed energy\/} $G_{\kappa,f}(\mu)$ of $\mu\in\mathcal E_\kappa(X)$
are formally given by
\begin{align}\label{wp}W^\mu_{\kappa,f}&:=U_\kappa^\mu+f,\\
\label{we}G_{\kappa,f}(\mu)&:=E_\kappa(\mu)+2\langle f,\mu\rangle=
\langle W^\mu_{\kappa,f}+f,\mu\rangle.\end{align} Let $\mathcal
E_{\kappa,f}(X)$ consist of all $\mu\in\mathcal E_\kappa(X)$ whose
$f$-weighted energy $G_{\kappa,f}(\mu)$ is finite, or equivalently
such that $f$ is $|\mu|$-integrable.

Given a set $Q\subset X$, let $\mathfrak M^+(Q;X)$ consist of all
$\mu\in\mathfrak M^+(X)$ {\it carried by\/} $Q$, which means that
$X\setminus Q$ is locally $\mu$-neg\-lig\-ible, or equivalently that
$Q$ is $\mu$-meas\-ur\-able and $\mu=\mu|_Q$, where
$\mu|_Q=1_Q\cdot\mu$ is the trace (restriction) of $\mu$ on $Q$
\cite[Chapter~V, Section~5, n$^\circ$\,3, Example]{B2}. (Here $1_Q$
denotes the indicator function of $Q$.) If $Q$ is closed, then $\mu$
is carried by $Q$ if and only if it is supported by $Q$, i.e.\
$S(\mu)\subset Q$. It follows from the countability of $X$ at
infinity that the concept of local $\mu$-neg\-lig\-ibility coincides
with that of $\mu$-neg\-lig\-ibility; and hence $\mu\in\mathfrak
M^+(Q;X)$ if and only if $\mu^*(X\setminus Q)=0$, $\mu^*(\cdot)$
being the {\it outer measure\/} of a set. Denoting by $\mu_*(\cdot)$
the {\it inner measure\/} of a set, for any $\mu\in\mathfrak
M^+(Q;X)$ we thus get
\[\mu^*(Q)=\mu_*(Q)=:\mu(Q).\]
Write $\mathcal E_\kappa^+(Q;X):=\mathcal E_\kappa(X)\cap\mathfrak M^+(Q;X)$,
$\mathfrak M^+(Q,q;X):=\bigl\{\mu\in\mathfrak M^+(Q;X):\ \mu(Q)=q\bigr\}$ and
$\mathcal E_\kappa^+(Q,q;X):=\mathcal E_\kappa(X)\cap\mathfrak M^+(Q,q;X)$, where $q\in(0,\infty)$.

Assume for a moment that $Q$ is {\it locally closed\/} in $X$.
According to \cite[Chapter~I, Section~3, Definition~2]{B1}, this
means that for every $x\in Q$ there is a neighborhood $V$ of $x$ in
$X$ such that $V\cap Q$ is a closed subset of the subspace $Q\subset
X$. Being locally closed, the set $Q$ is universally measurable
\cite[Chapter~I, Section~3, Proposition~5]{B1}, and hence $\mathfrak
M^+(Q;X)$ consists of all the restrictions $\mu|_{Q}$, $\mu$ ranging
over $\mathfrak M^+(X)$. On the other hand, according to
\cite[Chapter~I, Section~9, Proposition~13]{B1} the locally closed
set $Q$ itself can be thought of as a locally compact subspace of
$X$. Thus $\mathfrak M^+(Q;X)$ consists, in fact, of all those
$\nu\in\mathfrak M^+(Q)$ for each of which there exists
$\widehat{\nu}\in\mathfrak M^+(X)$ with the property
\begin{equation}\label{extend}\widehat{\nu}(\varphi)=\langle\varphi|_Q,\nu\rangle\text{ \ for every\ }
\varphi\in C_0(X).\end{equation}
We say that such $\widehat{\nu}$ {\it extends\/} $\nu\in\mathfrak M^+(Q)$ by $0$ off $Q$ to all of $X$.
A sufficient condition for this to happen is that $\nu$ be bounded.

In all that follows a kernel $\kappa$ is assumed to be {\it strictly
positive definite\/}, which means that the energy $E_\kappa(\mu)$,
$\mu\in\mathfrak M(X)$, is nonnegative whenever defined and it
equals $0$ only for $\mu=0$. Then $\mathcal E_\kappa(X)$ forms a
pre-Hil\-bert space with the inner product $E_\kappa(\mu,\mu_1)$ and
the energy norm $\|\mu\|_\kappa:=\sqrt{E_\kappa(\mu)}$, see
\cite{F1}. The (Hausdorff) topology on $\mathcal E_\kappa(X)$
determined by the norm $\|\cdot\|_\kappa$ is termed {\it strong\/}.

In contrast to \cite{Fu4,Fu5} where a capacity has been treated as a
functional acting on positive numerical functions on $X$, in the
present study we use the (standard) concept of capacity as a set
function. Thus the ({\it inner\/}) {\it capacity\/} of a set
$Q\subset X$ relative to the kernel $\kappa$, denoted $c_\kappa(Q)$,
is defined by
\begin{equation}\label{cap-def}c_\kappa(Q):=\bigl[\inf_{\mu\in\mathcal
E_\kappa^+(Q,1;X)}\,E_\kappa(\mu)\bigr]^{-1},\end{equation}
see e.g.\ \cite{F1,O}. Then $0\leqslant c_\kappa(Q)\leqslant\infty$. (As usual, here and in the sequel the
infimum over the empty set is taken to be $+\infty$. We also put
$1\bigl/(+\infty)=0$ and $1\bigl/0=+\infty$.) In consequence of the strict positive definiteness of the
kernel $\kappa$,
\begin{equation}\label{compact-fin}c_\kappa(K)<\infty\text{ \ for every compact\ }K\subset X.\end{equation}
Furthermore, by \cite[p.~153]{F1},
\begin{equation}\label{compact}c_\kappa(Q)=\sup\,c_\kappa(K)\quad(K\subset Q, \ K\text{\ compact}).
\end{equation}

An assertion $\mathcal U(x)$ involving a variable point $x\in X$ is
said to hold $c_\kappa$-{\it near\-ly everywhere\/} ($c_\kappa$-{\it
n.e.\/}) on $Q$ if $c_\kappa(N)=0$, where $N$ consists of all $x\in
Q$ for which $\mathcal U(x)$ fails. Throughout the paper we shall
often use the fact that $c_\kappa(N)=0$ if and only if $\mu_*(N)=0$
for every $\mu\in\mathcal E_\kappa^+(X)$, see
\cite[Lemma~2.3.1]{F1}.

As in \cite[p.\ 134]{L}, we call a (signed Radon) measure
$\mu\in\mathfrak M(X)$ $c_\kappa${\it -ab\-sol\-ut\-ely
continuous\/} if $\mu(K)=0$ for every compact set $K\subset X$ with
$c_\kappa(K)=0$. It follows from (\ref{compact}) that for such
$\mu$, $|\mu|_*(Q)=0$ for every $Q\subset X$ with $c_\kappa(Q)=0$.
Hence, every $\mu\in\mathcal E_\kappa(X)$ is
$c_\kappa$-ab\-sol\-utely continuous; but not conversely, see
\cite[pp.~134--135]{L}.

\begin{definition}\label{def-perf}Following~\cite{F1}, we call a (strictly positive definite)
kernel $\kappa$ {\it perfect\/} if every strong Cauchy sequence in $\mathcal E_\kappa^+(X)$ converges
strongly to any of its vague cluster points.\footnote{It follows from Theorem~\ref{fu-complete} that
for a perfect kernel such a vague cluster point exists and is unique.}\end{definition}

\begin{remark}\label{rem:clas}On $X=\mathbb R^n$, $n\geqslant3$, the $\alpha$-Riesz kernel
$\kappa_\alpha(x,y)=|x-y|^{\alpha-n}$, $\alpha\in(0,n)$, is strictly
positive definite and moreover perfect \cite{D1,D2}, and hence so is
the Newtonian kernel $\kappa_2(x,y)=|x-y|^{2-n}$ \cite{Car}.
Recently it has been shown that if $X$ is an open set $D$ in
$\mathbb R^n$, $n\geqslant3$, and $g^\alpha_D$, $\alpha\in(0,2]$, is
the $\alpha$-Green kernel on $D$ \cite[Chapter~IV, Section~5]{L},
then $\kappa=g^\alpha_D$ likewise is strictly positive definite and
moreover perfect \cite[Theorems~4.9, 4.11]{FZ}.\end{remark}

\begin{theorem}[{\rm see \cite{F1}}]\label{fu-complete} If a kernel\/ $\kappa$ on a locally compact space\/
$X$ is perfect, then the cone\/ $\mathcal E_\kappa^+(X)$ is strongly complete and the strong topology on\/
$\mathcal E_\kappa^+(X)$ is finer than the\/ {\rm(}induced\/{\rm)} vague topology on\/
$\mathcal E_\kappa^+(X)$.\end{theorem}

\begin{remark}\label{remma}In contrast to Theorem~\ref{fu-complete}, for a perfect kernel $\kappa$
the whole pre-Hilbert space $\mathcal E_\kappa(X)$ is in general
strongly {\it incomplete\/}, and this is the case even for the
$\alpha$-Riesz kernel of order $\alpha\in(1,n)$ on $\mathbb R^n$,
$n\geqslant 3$ (see \cite{Car} or \cite[Theorem~1.19]{L}). When
speaking of a completion of $\mathcal E_{\kappa_\alpha}(\mathbb
R^n)$, one needs to consider e.g.\ tempered distributions of finite
Deny–-Schwartz energy defined with the aid of the Fourier transform
\cite{D1}. Recently it has also been shown that if we restrict
ourselves to $\nu\in\mathcal E_{\kappa_\alpha}(\mathbb R^n)$ such
that $S_{\mathbb R^n}^\nu\subset D$, $D$ being a {\it bounded\/}
domain in $\mathbb R^n$, then the pre-Hilbert space of all those
$\nu$ can be isometrically imbedded into its completion, the Sobolev
space $\widetilde{H}^{-\alpha/2}(D)$, see \cite[Corollary~3.3]{HWZ}.
\end{remark}

\begin{remark}\label{remark}The concept of perfect kernel is an efficient tool in minimum energy problems
over classes of {\it positive scalar\/} Radon measures with finite
energy. Indeed, if $Q\subset X$ is closed,
$c_\kappa(Q)\in(0,\infty)$, and $\kappa$ is perfect, then the
minimum energy problem (\ref{cap-def}) has a unique solution
$\lambda_Q$, termed the ({\it inner\/}) $\kappa$-{\it capacitary
measure\/} on~$Q$ \cite[Theorem~4.1]{F1}. Later the concept of
perfectness has been shown to be efficient also in minimum energy
problems over classes of {\it vector measures\/} of finite or
infinite dimensions associated with a standard condenser, see
\cite{ZPot1}--\cite{ZPot3}. The approach developed in
\cite{ZPot1}--\cite{ZPot3} used substantially the assumption of the
boundedness of the kernel on the product of the oppositely charged
plates of a condenser, which made it possible to extend Cartan's
proof \cite{Car} of the strong completeness of the cone $\mathcal
E_{\kappa_2}^+(\mathbb R^n)$ of all positive measures on $\mathbb
R^n$ with finite Newtonian energy to an arbitrary perfect kernel
$\kappa$ on a locally compact space $X$ and suitable classes of {\it
signed\/} measures $\mu\in\mathcal E_\kappa(X)$; compare with
Remark~\ref{remma} above.\end{remark}

\section{$\alpha$-Riesz balayage and $\alpha$-Green function}\label{sec:RG}
In all that follows fix $n\geqslant3$, $\alpha\in(0,2]$ and a domain
$D\subset\mathbb R^n$ with $c_{\kappa_\alpha}(D^c)>0$, where
$D^c:=\mathbb R^n\setminus D$, and assume that either
$\kappa=\kappa_\alpha$ is the $\alpha$-{\it Riesz kernel\/} on
$X=\mathbb R^n$, or $\kappa=g_D^\alpha$ is the $\alpha$-{\it Green
kernel\/} on $X=D$ \cite[Chapter~IV, Section~5]{L} (or see below).
We simply write $\alpha$ instead of $\kappa_\alpha$ if
$\kappa_\alpha$ serves as an index, and we use the short form `n.e.'
instead of `$c_\alpha$-n.e.' if this will not cause any
mis\-under\-standing.

Given $x\in\mathbb R^n$ and $r\in(0,\infty)$, write
$B(x,r):=\{y\in\mathbb R^n:\ |y-x|<r\}$, $S(x,r):=\{y\in\mathbb
R^n:\ |y-x|=r\}$ and $\overline{B}(x,r):=B(x,r)\cup S(x,r)$.
Throughout the paper $\partial Q$ denotes the boundary of a set
$Q\subset\mathbb R^n$ in the topology of~$\mathbb R^n$.

When speaking of a positive Radon measure $\mu$ on $\mathbb R^n$, we
always tacitly assume that $U_\alpha^\mu$ is not identically
infinite. This implies that
\begin{equation}\label{1.3.10}\int_{|y|>1}\,\frac{d\mu(y)}{|y|^{n-\alpha}}<\infty,\end{equation}
see \cite[Eq.~(1.3.10)]{L}, and consequently that $U_\alpha^\mu$ is
finite ($c_\alpha$-)n.e.\ on $\mathbb R^n$ \cite[Chap\-ter~III,
Section~1]{L}; these two implications can actually be reversed.

\begin{definition}\label{d-ext} $\nu\in\mathfrak M(D)$ is called {\it extendible\/} if there exist
$\widehat{\nu^+}$ and $\widehat{\nu^-}$ extending $\nu^+$ and
$\nu^-$, respectively, by $0$ off $D$ to all of $\mathbb R^n$, see
(\ref{extend}), and if these $\widehat{\nu^+}$ and $\widehat{\nu^-}$
satisfy (\ref{1.3.10}). We identify such $\nu\in\mathfrak M(D)$ with
its extension $\widehat\nu:=\widehat{\nu^+}-\widehat{\nu^-}$, and we
therefore write $\widehat\nu=\nu$.\end{definition}

Every bounded measure $\nu\in\mathfrak M(D)$ is extendible. The converse holds if $D$ is bounded, but not in
general (e.g.\ not if $D^c$ is compact). The set of all extendible measures $\nu\in\mathfrak M(D)$ consists
of all the restrictions $\mu|_D$ where $\mu$ ranges over $\mathfrak M(\mathbb R^n)$.

The $\alpha$-{\it Green kernel\/} $g=g_D^\alpha$ on $D$ is defined by
\begin{equation*}g^\alpha_D(x,y)=U_\alpha^{\varepsilon_y}(x)-
U_\alpha^{\varepsilon_y^{D^c}}(x)\text{ \ for all\ }x,y\in
D,\end{equation*} where $\varepsilon_y$ denotes the unit Dirac
measure at a point $y$ and $\varepsilon_y^{D^c}$ its $\alpha$-{\it
Riesz balayage\/} onto the (closed) set $D^c$, determined uniquely
in the frame of the classical approach by \cite[Theorem~3.6]{FZ}.
See also the book by Bliedtner and Hansen \cite{BH} where balayage
is studied in the setting of balayage spaces.

We shall simply write $\mu'$ instead of $\mu^{D^c}$ when speaking of
the $\alpha$-Riesz balayage of $\mu\in\mathfrak M^+(D;\mathbb R^n)$
onto $D^c$. According to \cite[Corollaries~3.19, 3.20]{FZ}, for any
$\mu\in\mathfrak M^+(D;\mathbb R^n)$ the balayage $\mu'$ is
$c_\alpha$-ab\-sol\-ut\-ely continuous, and it is determined
uniquely by the relation
\begin{equation}\label{bal-eq}U_\alpha^{\mu'}=U_\alpha^{\mu}\text{ \ n.e.\ on\ }D^c\end{equation}
among the $c_\alpha$-absolutely continuous positive measures on $\mathbb R^n$ supported by $D^c$.
Furthermore, there holds the integral representation
\begin{equation}\label{int-repr}\mu'=\int\varepsilon_y'\,d\mu(y),\end{equation}
see \cite[Theorem~3.17]{FZ}.\footnote{In the literature the integral
representation (\ref{int-repr}) seems to have been more or less
taken for granted, though it has been pointed out in
\cite[Chapter~V, Section~3, n$^\circ$\,1]{B2} that it requires that
the family $(\varepsilon_y')_{y\in D}$ be $\mu${\it -ad\-equ\-ate\/}
in the sense of \cite[Chapter~V, Section~3, Definition~1]{B2}; see
also counterexamples (without $\mu$-ad\-equ\-acy) in Exercises~1
and~2 at the end of that section. A proof of this adequacy  has
therefore been given in \cite[Lemma~3.16]{FZ}.} If moreover
$\mu\in\mathcal E_\alpha^+(D;\mathbb R^n)$, then the balayage $\mu'$
is in fact the orthogonal projection of $\mu$ onto the convex cone
$\mathcal E^+_\alpha(D^c;\mathbb R^n)$, i.e.\ $\mu'\in\mathcal
E^+_\alpha(D^c;\mathbb R^n)$ and
\begin{equation}\label{proj}\|\mu-\theta\|_\alpha>\|\mu-\mu'\|_\alpha\text{ \ for all \ }
\theta\in\mathcal E^+_\alpha(D^c;\mathbb R^n), \
\theta\ne\mu'\end{equation} (see \cite[Theorem~4.12]{Fu5} or
\cite[Theorem~3.1]{FZ}).

If now $\nu\in\mathfrak M(D)$ is an extendible ({\it signed\/}
Radon) measure, then $\nu':=\nu^{D^c}:=(\nu^+)'-(\nu^-)'$ is said to
be a {\it balayage\/} of $\nu$ onto $D^c$. It follows from
\cite[Chapter~III, Section~1, n$^\circ$\,1, Remark]{L} that the
balayage $\nu'$ is determined uniquely by (\ref{bal-eq}) with $\nu$
in place of $\mu$ among the $c_\alpha$-absolutely continuous signed
measures on $\mathbb R^n$ supported by~$D^c$.

The following definition goes back to Brelot
\cite[Theorem~VII.13]{Brelo2}. A closed set $F\subset\mathbb R^n$ is
said to be {\it $\alpha$-thin at infinity\/} if either $F$ is
compact, or the inverse of $F$ relative to $S(0,1)$ has $x=0$ as an
$\alpha$-irregular boundary point (cf.\ \cite[Theorem~5.10]{L}).

\begin{theorem}[{\rm see \cite[Theorem~3.22]{FZ}}]\label{bal-mass-th} The set\/ $D^c$ is not\/
$\alpha$-thin at infinity if and only if for every bounded measure\/
$\mu\in\mathfrak M^+(D)$ we have\/ $\mu'(\mathbb R^n)=\mu(\mathbb
R^n)$.\footnote{In general, $\nu^{D^c}(\mathbb
R^n)\leqslant\nu(\mathbb R^n)$ for every $\nu\in\mathfrak
M^+(\mathbb R^n)$ \cite[Theorem~3.11]{FZ}.}
\end{theorem}

As noted in Remark~\ref{rem:clas}, the $\alpha$-Riesz kernel
$\kappa_\alpha$ on $\mathbb R^n$ as well as the $\alpha$-Green
kernel $g^\alpha_D$ on $D$ is strictly positive definite and
moreover perfect. Furthermore, the kernel $\kappa_\alpha$ (with
$\alpha\in(0,2]$) satisfies the complete maximum principle in the
form stated in \cite[Theorems~1.27, 1.29]{L}. Regarding a similar
result for the kernel $g$, the following assertion holds.

\begin{theorem}[{\rm see \cite[Theorem~4.6]{FZ}}]\label{th-dom-pr} Let\/ $\mu\in\mathcal E^+_g(D)$, let\/
$\nu\in\mathfrak M^+(D)$ be extendible, and let\/ $v$ be a positive\/ $\alpha$-super\-har\-monic function
on\/ $\mathbb R^n$ {\rm({\it see\/} \cite[{\it Chapter\/}~I, {\it Section\/}~5, {\it n\/}$^\circ$\,20]{L})}.
If moreover\/ $U_g^\mu\leqslant U_g^\nu+v$ $\mu$-a.e.\ on $D$, then the same inequality holds on all
of\/~$D$.\end{theorem}

The following three lemmas establish relations between potentials
and energies relative to the kernels $\kappa_\alpha$ and
$g=g^\alpha_D$.

\begin{lemma}\label{l-hatg} For any extendible measure\/ $\mu\in\mathfrak M(D)$ the $\alpha$-Green
potential\/ $U_g^{\mu}$ is finite\/ {\rm(}$c_\alpha$-{\rm)}n.e.\
on\/ $D$ and given by\/\footnote{If $Q$ is a given subset of $D$,
then any assertion involving a variable point holds n.e.\ on $Q$ if
and only if it holds $c_g$-n.e.\ on $Q$, see
\cite[Lemma~2.6]{DFHSZ}.\label{RG}}
\begin{equation}\label{hatg}U_g^{\mu}=U_\alpha^{\mu-\mu'}\text{ \ n.e.\ on\ }D.\end{equation}
\end{lemma}

\begin{proof}It is seen from Definition~\ref{d-ext} that $U_\alpha^\mu$ is finite n.e.\ on $\mathbb R^n$,
and hence so is $U_\alpha^{\mu'}$. Applying  (\ref{int-repr}) to $\mu^\pm$, we get by
\cite[Chapter~V, Section~3, Theorem~1]{B2}
 \[U_g^\mu=\int\,\bigl[U_\alpha^{\varepsilon_y}-U_\alpha^{\varepsilon_y'}\bigr]\,d\mu(y)=
 U_\alpha^\mu-U_\alpha^{\mu'}\]
n.e.\ on $D$, as was to be proved.\end{proof}

\begin{lemma}\label{l-hen'} If\/ $\mu\in\mathfrak M(D)$ is extendible and its extension belongs
to\/ $\mathcal E_\alpha(\mathbb R^n)$, then
\begin{align}
\label{l3-1}&\mu\in\mathcal E_g(D),\\
\label{l3-2}&\mu-\mu'\in\mathcal E_\alpha(\mathbb R^n),\\
\label{eq1-2-hen}&\|\mu\|^2_g=\|\mu-\mu'\|^2_\alpha=\|\mu\|^2_\alpha-\|\mu'\|^2_\alpha.
\end{align}
\end{lemma}

\begin{proof}In view of the definition of a signed measure of finite energy (see Section~\ref{sec:princ}),
we obtain (\ref{l3-1}) from the inequality\footnote{The strict
inequality in (\ref{g-ineq}) is caused by our convention that
$c_\alpha(D^c)>0$.}
\begin{equation}\label{g-ineq}g^\alpha_D(x,y)<\kappa_\alpha(x,y)\text{ \ for all \ }x,y\in D,\end{equation}
while (\ref{l3-2}) from \cite[Corollary~3.7]{FZ} or \cite[Theorems~3.1, 3.6]{FZ}.
According to Lemma~\ref{l-hatg} and footnote~\ref{RG}, $U_g^\mu$ is finite $c_g$-n.e.\ on $D$ and given
by (\ref{hatg}), while by (\ref{l3-1}) the same holds $|\mu|$-a.e.\ on $D$, see \cite[Lemma~2.3.1]{F1}.
Integrating (\ref{hatg}) with respect to $\mu^\pm$, we therefore obtain by subtraction
\begin{equation}\label{l1}\infty>E_g(\mu)=E_\alpha(\mu-\mu',\mu).\end{equation}
As $U_\alpha^{\mu-\mu'}=0$ n.e.\ on $D^c$ by (\ref{bal-eq}), while $\mu'$ is $c_\alpha$-absolutely continuous,
we also have
\begin{equation}\label{l2}E_\alpha(\mu-\mu',\mu')=0,\end{equation}
which results in the former equality in (\ref{eq1-2-hen}) when
combined with (\ref{l1}). In view of (\ref{l3-2}), (\ref{l2}) takes
the form $\|\mu'\|^2_\alpha=E_\alpha(\mu,\mu')$, and the former
equality in (\ref{eq1-2-hen}) therefore implies the
latter.\end{proof}

\begin{lemma}\label{l-hen'-comp}
Assume that\/ $\mu\in\mathfrak M(D)$ has compact support\/
$S^\mu_D$. Then\/ $\mu\in\mathcal E_g(D)$ if and only if its
extension belongs to\/ $\mathcal E_\alpha(\mathbb R^n)$.\footnote{If
the measure in question is positive, then Lemma~\ref{l-hen'-comp}
can be generalized to any bounded $\mu\in\mathfrak M^+(D)$ such that
the Euclidean distance between $S^\mu_D$ and $\partial D$ is ${}>0$,
see \cite[Lemma~3.4]{FZ-Pot}.}\end{lemma}

\begin{proof}According to Lemma~\ref{l-hen'}, it is enough to establish the necessity part of the lemma.
We may clearly assume that $\mu$ is positive. Since $U_\alpha^{\mu'}$ is continuous on $D$, and hence bounded
on the compact set $S_D^{\mu}$, we have
\begin{equation}\label{comp-f-e}E_\alpha(\mu,\mu')<\infty.\end{equation} On the other hand,
$E_g(\mu)$ is finite by assumption, and hence likewise as in the
preceding proof relation (\ref{l1}) holds. Combining (\ref{l1}) with
(\ref{comp-f-e}) yields $\mu\in\mathcal E_\alpha(\mathbb R^n)$.
\end{proof}

\begin{remark}\label{rem:in}The proof of Lemma~\ref{l-hen'} uses substantially the requirement
$\mu\in\mathcal E_\alpha(\mathbb R^n)$. Being founded on the weaker
assumption $\mu\in\mathcal E_g(D)$, a similar assertion in
\cite{DFHSZ} (see Lemma~2.4 there) was incorrect, as will be shown
by Example~\ref{counterex} below. The revision of \cite{DFHSZ}
provided in present paper is based significantly on the current
version of Lemma~\ref{l-hen'} as well as on the perfectness of the
kernel $g^\alpha_D$, discovered recently in
\cite[Theorem~4.11]{FZ}.\end{remark}

\section{Minimum $\alpha$-Riesz energy problems for generalized condensers}\label{sec:pr1}

\subsection{A generalized condenser}\label{sec:gen1}Under the (permanent) assumptions stated at the beginning
of Section~\ref{sec:RG}, fix a (not necessarily proper) subset $A_1$
of $D$ which is relatively closed in $D$. The pair $\mathbf
A=(A_1,A_2)$, where $A_2:=D^c$, is said to form a {\it generalized
condenser\/} in $\mathbb R^n$, and $A_1$ and $A_2$ are termed its
{\it positive\/} and {\it negative plates\/}.\footnote{The notion of
generalized condenser thus defined differs from that introduced in
our recent work \cite{DFHSZ2}; cf.\ Remark~\ref{J} below.} To avoid
triviality, we shall always require that $c_\alpha(A_1)>0$, and
hence
\begin{equation}\label{nonzero'}c_\alpha(A_i)>0\text{ \ for\ }i=1,2.\end{equation}
The generalized condenser $\mathbf A=(A_1,A_2)$ is said to be {\it standard\/} if $A_1$ is closed in
$\mathbb R^n$.

\begin{example}\label{ex-1}Let $A_1=B(0,r)=D$, $r\in(0,\infty)$. Then $\mathbf A=(A_1,A_2)$ is a generalized
condenser in $\mathbb R^n$, which certainly is not standard. See
Example~\ref{ex-2} for constraints under which the constrained
minimum $\alpha$-Riesz energy problem (Problem~\ref{prR}) for such
$\mathbf A$ admits a solution (has no short-circ\-uit) despite the
fact that $A_2\cap\mathrm{Cl}_{\mathbb R^n}A_1=S(0,r)$.
\end{example}

Unless explicitly stated otherwise, in all that follows $\mathbf
A=(A_1,A_2)$ is assumed to be a generalized condenser in $\mathbb
R^n$. We emphasize that, though $A_1\cap A_2=\varnothing$, the set
$A_2\cap\mathrm{Cl}_{\mathbb R^n}A_1$ may have nonzero
$\alpha$-Riesz capacity and may even coincide with the
whole~$\partial D$.

Let $\mathfrak M(\mathbf A;\mathbb R^n)$ consist of all (signed
Radon) measures on $\mathbb R^n$ whose positive and negative parts
in the Hahn--Jor\-dan decomposition are carried by $A_1$ and $A_2$,
respectively, and let $\mathcal E_\alpha(\mathbf A;\mathbb
R^n):=\mathfrak M(\mathbf A;\mathbb R^n)\cap\mathcal
E_\alpha(\mathbb R^n)$. For any vector $\mathbf a=(a_1,a_2)$ with
$a_1,a_2>0$ write
\[\mathcal E_\alpha(\mathbf A,\mathbf a;\mathbb R^n):=
\bigl\{\mu\in\mathcal E_\alpha(\mathbf A;\mathbb R^n): \
\mu^+(A_1)=a_1, \ \mu^-(A_2)=a_2\bigr\}.\] This class is nonempty,
which is clear from (\ref{nonzero'}) by \cite[Lemma~2.3.1]{F1}, and
it therefore makes sense to consider the problem on the existence of
$\lambda_{\mathbf A}\in\mathcal E_\alpha(\mathbf A,\mathbf a;\mathbb
R^n)$ with
\begin{equation}\label{pr-cond}\|\lambda_{\mathbf A}\|^2_\alpha=
w_\alpha(\mathbf A,\mathbf a):=
\inf_{\mu\in\mathcal E_\alpha(\mathbf A,\mathbf a;\mathbb R^n)}\,\|\mu\|^2_\alpha.\end{equation}
This problem will be referred to as the {\it condenser problem\/}. By the (strict) positive definiteness of
the kernel $\kappa_\alpha$,
\[w_\alpha(\mathbf A,\mathbf a)\geqslant0.\]

\begin{remark}Assume for a moment that $\mathbf A$ is a standard condenser in $\mathbb R^n$. If moreover
it possesses the separation property
\begin{equation}\label{dist}\inf_{(x,y)\in A_1\times A_2}\,|x-y|>0,\end{equation}
then the assumption
\begin{equation}\label{cap-f}c_\alpha(A_i)<\infty\text{ \ for\ }i=1,2\end{equation}
is sufficient for problem (\ref{pr-cond}) to be (uniquely) solvable
for every normalizing vector $\mathbf a$. See e.g.\ \cite{ZPot2}
where this result has actually been established even for infinite
dimensional vector measures in the presence of a vector-valued
external field and for an arbitrary perfect kernel on a locally
compact space. However, if (\ref{cap-f}) fails to hold, then in
general there exists a vector $\mathbf a'$ such that the
corresponding extremal value $w_\alpha(\mathbf A,\mathbf a')$ is
{\it not\/} an actual minimum, see \cite{ZPot2}.\footnote{In the
case of the $\alpha$-Riesz kernels of order $1<\alpha\leqslant2$ on
$\mathbb R^3$ some of the (theoretical) results on the solvability
or unsolvability of the condenser problem, mentioned in
\cite{ZPot2}, have been illustrated in \cite{HWZ,OWZ} by means of
numerical experiments.} Therefore it was interesting to give a
description of the set of all vectors $\mathbf a$ for which the
condenser problem nevertheless is solvable. Such a characterization
has been established in \cite{ZPot3}. On the other hand,  if the
separation condition (\ref{dist}) is omitted, then the approach
developed in \cite{ZPot2,ZPot3} breaks down and (\ref{cap-f}) does
not guarantee anymore the existence of a solution to problem
(\ref{pr-cond}). This has been illustrated by
\cite[Theorem~4.6]{DFHSZ2} pertaining to the Newtonian
kernel.\end{remark}

The following theorem shows that for a generalized condenser $\mathbf A$ the condenser problem in general
has no solution. Denote $\mathbf 1:=(1,1)$.

\begin{theorem}\label{pr1uns}If\/ $A_2$ is not\/ $\alpha$-thin at infinity and\/
$c_g(A_1)=\infty$, then
\[w_\alpha(\mathbf A,\mathbf 1)=\bigl[c_g(A_1)\bigr]^{-1}=0.\]
Hence, $w_\alpha(\mathbf A,\mathbf 1)$ cannot be an actual minimum because\/
$0\notin\mathcal E_\alpha(\mathbf A,\mathbf 1;\mathbb R^n)$.\end{theorem}

\begin{proof} Consider an exhaustion of $A_1$ by an increasing sequence $\{K_j\}_{j\in\mathbb N}$ of
compact sets. By~(\ref{compact}),
\begin{equation}\label{incr}c_g(K_j)\uparrow c_g(A_1)=\infty\text{ \ as\ }j\to\infty,\end{equation}
and there is therefore no loss of generality in assuming that every
$c_g(K_j)$ is ${}>0$. Furthermore, since the $\alpha$-Green kernel
$g$ is strictly positive definite and moreover perfect
(Remark~\ref{rem:clas}), we see from (\ref{compact-fin}) that
$c_g(K_j)<\infty$ and hence, by Remark~\ref{remark}, there exists a
(unique) $g$-capacitary measure $\lambda_j$ on $K_j$, i.e.\
$\lambda_j\in\mathcal E_g^+(K_j,1;D)$ with
\[\|\lambda_j\|^2_g=1/c_g(K_j)<\infty.\]
According to Lemma~\ref{l-hen'-comp}, $E_\alpha(\lambda_j)$ is
finite along with $E_g(\lambda_j)$ and hence, by Lemma~\ref{l-hen'},
\[\|\lambda_j\|^2_g=\|\lambda_j-\lambda_j'\|^2_\alpha.\]
As $A_2$ is not $\alpha$-thin at infinity, we see from
Theorem~\ref{bal-mass-th} that $\lambda_j-\lambda_j'\in\mathcal
E_\alpha(\mathbf A,\mathbf 1;\mathbb R^n)$, which together with the
two preceding displays yields
\[1/c_g(K_j)=\|\lambda_j\|^2_g=\|\lambda_j-\lambda_j'\|^2_\alpha\geqslant
w_\alpha(\mathbf A,\mathbf 1)\geqslant0.\] Letting here
$j\to\infty$, we obtain the theorem from (\ref{incr}).
\end{proof}

Using the electrostatic interpretation, which is possible for the
Coulomb kernel $|x-y|^{-1}$ on $\mathbb R^3$, we say that under the
hypotheses of Theorem~\ref{pr1uns} a short-circ\-uit occurs between
the oppositely charged plates of the generalized condenser $\mathbf
A$. It is therefore meaningful to ask what kinds of additional
requirements on the objects in question will prevent this blow-up
effect, and secure that a solution to the corresponding minimum
$\alpha$-Riesz energy problem does exist. To this end we have
succeeded in working out a substantive theory by imposing a suitable
upper constraint on the measures under consideration, thereby
rectifying the results on the constrained $\alpha$-Riesz energy
problem announced in \cite{DFHSZ}, cf.\ Remark~\ref{rem:in} above.

\subsection{A constrained $f$-weighted minimum $\alpha$-Riesz energy problem for a generalized
condenser}\label{sec-def} In the rest of the paper we shall always
require that $A_2$ {\it is not\/ $\alpha$-thin at infinity\/} and
that $\mathbf a=\mathbf 1$. When speaking of an external field $f$,
see Section~\ref{sec:princ}, we shall tacitly assume that either of
the following Case~I or Case~II holds:
\begin{itemize}
\item[\rm I.] {\it $f\in\Psi(\mathbb R^n)$ and moreover}
\begin{equation}\label{fp}f=0\text{\it \  n.e.\ on\ } A_2;\end{equation}
\item[\rm II.] {\it $f=U_{\alpha}^{\zeta-\zeta'}$, where\/ $\zeta$ is a signed extendible Radon measure
on\/ $D$ with\/ $E_\alpha(\zeta)<\infty$}.
\end{itemize}
Note that relation (\ref{fp}) holds also in Case II, see (\ref{bal-eq}). Since a set with
$c_\alpha(\cdot)=0$ carries no measure with finite $\alpha$-Riesz energy \cite[Lemma~2.3.1]{F1}, we thus see
that in either Case~I or Case~II {\it no external field acts on the measures from\/}
$\mathcal E^+_\alpha(A_2;\mathbb R^n)$. The $f$-weighted $\alpha$-Riesz energy $G_{\alpha,f}(\mu)$,
cf.\ (\ref{we}), of $\mu\in\mathcal E_\alpha(\mathbf A;\mathbb R^n)$ can therefore be defined as
\begin{equation}\label{EG}G_{\alpha,f}(\mu):=\|\mu\|^2_\alpha+2\langle f,\mu\rangle=
\|\mu\|^2_\alpha+2\langle f,\mu^+\rangle.\end{equation}
If Case II takes place, then for every $\mu\in\mathcal E_\alpha(\mathbf A;\mathbb R^n)$ we moreover get
\begin{align}\label{GCII}\infty>G_{\alpha,f}(\mu)&=\|\mu\|^2_\alpha+2E_\alpha(\zeta-\zeta',\mu)\\
{}&=\|\mu+\zeta-\zeta'\|^2_\alpha-\|\zeta-\zeta'\|_\alpha^2\geqslant-\|\zeta-\zeta'\|_\alpha^2>-\infty.
\notag\end{align} Thus in either Case I or Case II,
\begin{equation}\label{Glowerb}G_{\alpha,f}(\mu)\geqslant-M>-\infty\text{ \ for all\ }
\mu\in\mathcal E_\alpha(\mathbf A;\mathbb R^n).\end{equation}
Indeed, in Case I this is obvious by (\ref{EG}), while in Case~II it follows from~(\ref{GCII}).

By a {\it constraint\/} for measures from $\mathcal
E^+_\alpha(A_1,1;\mathbb R^n)$ we mean any $\xi$ such that
\begin{equation}\label{constr1}\xi\in\mathcal E^+_\alpha(A_1;\mathbb R^n)\text{ \ and \ }\xi(A_1)>1.
\end{equation}
Let $\mathfrak C(A_1;\mathbb R^n)$ consist of all such constraints.
Given $\xi\in\mathfrak C(A_1;\mathbb R^n)$, write
\[\mathcal E_\alpha^\xi(\mathbf A,\mathbf 1;\mathbb R^n):=
\bigl\{\mu\in\mathcal E_\alpha(\mathbf A,\mathbf 1;\mathbb R^n): \
\mu^+\leqslant\xi\bigr\},\] where $\mu^+\leqslant\xi$ means that
$\xi-\mu^+\geqslant0$. Note that {\it we do not impose any
constraint on the negative parts of measures\/} $\mu\in\mathcal
E_\alpha(\mathbf A,\mathbf 1;\mathbb R^n)$. If
\[\mathcal E_{\alpha,f}^\xi(\mathbf A,\mathbf 1;\mathbb R^n):=
\mathcal E_\alpha^\xi(\mathbf A,\mathbf 1;\mathbb R^n)\cap\mathcal
E_{\alpha,f}(\mathbb R^n)\ne\varnothing\] (see
Section~\ref{sec:princ} for the definition of the class $\mathcal
E_{\alpha,f}(\mathbb R^n)$), or equivalently if\footnote{If
(\ref{cggen}) is fulfilled, then $G_{\alpha,f}^{\xi}(\mathbf
A,\mathbf 1;\mathbb R^n)$ is actually finite, see
(\ref{Glowerb}).\label{foot-G-finite}}
\begin{equation}\label{cggen}G_{\alpha,f}^{\xi}(\mathbf A,\mathbf 1;\mathbb R^n):=
\inf_{\mu\in\mathcal E_{\alpha,f}^\xi(\mathbf A,\mathbf 1;\mathbb R^n)}\,G_{\alpha,f}(\mu)<\infty,
\end{equation}
then the following {\it constrained\/ $f$-weighted minimum\/ $\alpha$-Riesz energy problem\/} makes sense.

\begin{problem}\label{prR}Does there exist
$\lambda_{\mathbf A}^\xi\in\mathcal E_{\alpha,f}^\xi(\mathbf A,\mathbf 1;\mathbb R^n)$ with
\begin{equation}\label{inf2}G_{\alpha,f}(\lambda_{\mathbf A}^\xi)=
G_{\alpha,f}^{\xi}(\mathbf A,\mathbf 1;\mathbb R^n)\,?\end{equation}
\end{problem}

Conditions which guarantee (\ref{cggen}) are provided by the following Lemma~\ref{suff-fin}. Write
\begin{equation}\label{Acirc}A_1^\circ:=\bigl\{x\in A_1: \ |f(x)|<\infty\bigr\}.\end{equation}

\begin{lemma}\label{suff-fin}
Relation\/ {\rm(\ref{cggen})} holds if either Case\/~{\rm II} takes place, or\/ {\rm(}in the presence of
Case\/~{\rm I}{\rm)} if
\begin{equation}\label{acirc}\xi(A_1^\circ)>1.\end{equation}
\end{lemma}

\begin{proof}Assume first that (\ref{acirc}) holds; then there exists by (\ref{Acirc}) a compact set
$K\subset A_1^\circ$ such that $|f|\leqslant M<\infty$ on $K$ and
$\xi(K)>1$. Define $\mu=\mu^+-\mu^-$, where
$\mu^+:=\xi|_{K}\bigl/\xi(K)$ while $\mu^-$ is any measure from
$\mathcal E^+_\alpha(A_2,1;\mathbb R^n)$ (such $\mu^-$ exists
because $c_\alpha(A_2)>0$). Noting that $\xi|_K\in\mathcal
E_\alpha^+(K;\mathbb R^n)$ by (\ref{constr1}), we get
$\mu\in\mathcal E_{\alpha,f}^\xi(\mathbf A,\mathbf 1;\mathbb R^n)$,
or equivalently (\ref{cggen}). To complete the proof of the lemma,
it is left to note that (\ref{acirc}) holds automatically whenever
Case II takes place, since then $U_\alpha^{\zeta-\zeta'}$ is finite
n.e.\ on $\mathbb R^n$, hence $\xi$-a.e.\ by~(\ref{constr1}).
\end{proof}

\begin{lemma}\label{l:uniq} A solution\/ $\lambda_{\mathbf A}^\xi$ to Problem\/~{\rm\ref{prR}} is unique\/
{\rm(}whenever it exists\/{\rm)}.\end{lemma}

\begin{proof}This can be established by standard methods based on the convexity of the class
$\mathcal E_{\alpha,f}^\xi(\mathbf A,\mathbf 1;\mathbb R^n)$ and the
pre-Hil\-bert structure on the space $\mathcal E_\alpha(\mathbb
R^n)$. Indeed, if $\lambda$ and $\breve{\lambda}$ are two solutions
to Problem~\ref{prR}, then
\[4G_{\alpha,f}^\xi(\mathbf A,\mathbf 1;\mathbb R^n)\leqslant
4G_{\alpha,f}\Bigl(\frac{\lambda+\breve{\lambda}}{2}\Bigr)=
\|\lambda+\breve{\lambda}\|_\alpha^2+4\langle f,\lambda+\breve{\lambda}\rangle.\]
On the other hand, applying the parallelogram identity in $\mathcal E_\alpha(\mathbb R^n)$ to $\lambda$ and
$\breve{\lambda}$ and then adding and
subtracting $4\langle f,\lambda+\breve{\lambda}\rangle$ we get
\[\|\lambda-\breve\lambda\|_\alpha^2=
-\|\lambda+\breve{\lambda}\|_\alpha^2-4\langle f,\lambda+\breve{\lambda}\rangle+
2G_{\alpha,f}(\lambda)+2G_{\alpha,f}(\breve{\lambda}).\]
When combined with the preceding relation, this yields
\[0\leqslant\|\lambda-\breve{\lambda}\|^2_\alpha
\leqslant-4G_{\alpha,f}^{\xi}(\mathbf A,\mathbf 1;\mathbb R^n)+2G_{\alpha,f}(\lambda)+
2G_{\alpha,f}(\breve{\lambda})=0.\]
Since $\|\cdot\|_\alpha$ is a norm, the lemma follows.\end{proof}

\section{Relations between minimum $\alpha$-Riesz and $\alpha$-Green energy problems}\label{deep}

We are keeping the (permanent) assumptions on $\mathbf A$, $f$ and $\xi$ stated in Sections~\ref{sec:gen1}
and~\ref{sec-def}. Since $\mathfrak M^+(A_1;\mathbb R^n)\subset\mathfrak M^+(A_1;D)$, the constraint $\xi$
can be thought of as an extendible measure from $\mathfrak M^+(A_1;D)$ such that its extension has finite
$\alpha$-Riesz energy (and total mass $\xi(A_1)>1$). Define
\[\mathcal E_g^\xi(A_1,1;D):=\bigl\{\mu\in\mathcal E_g^+(A_1,1;D): \ \mu\leqslant\xi\bigr\},\]
and let $\mathcal E_{g,f}^\xi(A_1,1;D)$ consist of all $\mu\in\mathcal E_g^\xi(A_1,1;D)$ such that
\begin{equation}\label{Gfdef}G_{g,f}(\mu):=G_{g,f|_D}(\mu)=\|\mu\|^2_g+2\langle f|_D,\mu\rangle\end{equation}
is finite, cf.\ (\ref{we}). We have used here the fact that
$\nu^*(D^c)=0$ for every $\nu\in\mathfrak M^+(D;\mathbb R^n)$, see
Section~\ref{sec:princ}. If the class $\mathcal
E_{g,f}^\xi(A_1,1;D)$ is nonempty, or equivalently if
\begin{equation}\label{Gfin}G_{g,f}^\xi(A_1,1;D):=
\inf_{\mu\in\mathcal E_{g,f}^\xi(A_1,1;D)}\,G_{g,f}(\mu)<\infty,\end{equation}
then the following {\it constrained\/ $f$-weighted minimum\/ $\alpha$-Green energy problem\/} makes sense.

\begin{problem}\label{prG}Does there exist $\lambda_{A_1}^\xi\in
\mathcal E_{g,f}^\xi(A_1,1;D)$ with
\begin{equation}\label{CF}G_{g,f}(\lambda_{A_1}^\xi)=G_{g,f}^{\xi}(A_1,1;D)\,?\end{equation}
\end{problem}

Based on the convexity of the class $\mathcal E_{g,f}^\xi(A_1,1;D)$
and the pre-Hilbert structure on the space $\mathcal E_g(D)$,
likewise as in the proof of Lemma~\ref{l:uniq} we see that {\it a
solution\/ $\lambda_{A_1}^\xi$ to Problem\/~{\rm\ref{prG}} is
unique\/} whenever it exists (see \cite[Lemma~4.1]{DFHSZ}).

\begin{theorem}\label{th:rel}Under the stated assumptions,
\begin{equation}\label{equality}G_{\alpha,f}^{\xi}(\mathbf A,\mathbf 1;\mathbb R^n)=G_{g,f}^{\xi}(A_1,1;D).
\end{equation}
Assume moreover that either of the\/ {\rm(}equivalent\/{\rm)} assumptions\/ {\rm(\ref{cggen})} or\/
{\rm(\ref{Gfin})} is fulfilled. Then the solution to Problem\/~{\rm\ref{prR}} exists if and only if
so does that to Problem\/~{\rm\ref{prG}}, and in the affirmative case they are related to each other
by the formula
\begin{equation}\label{reprrr}
\lambda^{\xi}_{\mathbf A}=\lambda^\xi_{A_1}-\bigl(\lambda^\xi_{A_1}\bigr)'.
\end{equation}
\end{theorem}

\begin{proof} We begin by establishing the inequality
\begin{equation}\label{ineq}G^\xi_{g,f}(A_1,1;D)\geqslant
G_{\alpha,f}^{\xi}(\mathbf A,\mathbf 1;\mathbb R^n).\end{equation}
Assuming $G^\xi_{g,f}(A_1,1;D)<\infty$, choose $\nu\in\mathcal E_{g,f}^\xi(A_1,1;D)$. Being bounded,
this $\nu$ is extendible. Furthermore, its extension has finite $\alpha$-Riesz energy, for so does
the extension of the constraint $\xi$ by (\ref{constr1}). Applying (\ref{eq1-2-hen}) and (\ref{Gfdef})
we get
\[G_{g,f}(\nu)=\|\nu-\nu'\|^2_\alpha+2\langle f|_D,\nu\rangle.\]
As $A_2$ is not $\alpha$-thin at infinity, we see from
Theorem~\ref{bal-mass-th} that $\theta:=\nu-\nu'\in\mathcal
E_\alpha^\xi(\mathbf A,\mathbf 1;\mathbb R^n)$. Furthermore, by
(\ref{EG}),
\[\langle f,\theta\rangle=\langle f|_D,\nu\rangle<\infty.\]
Thus $\theta\in\mathcal E^\xi_{\alpha,f}(\mathbf A,\mathbf 1;\mathbb
R^n)$ and $G_{\alpha,f}(\theta)=G_{g,f}(\nu)$, the latter relation
being valid according to the two preceding displays. This yields
\begin{equation}\label{again}G_{g,f}(\nu)=G_{\alpha,f}(\theta)\geqslant
G^{\xi}_{\alpha,f}(\mathbf A,\mathbf 1;\mathbb R^n),\end{equation}
which establishes (\ref{ineq}) by letting here $\nu$ range over
$\mathcal E^{\xi}_{g,f}(A_1,1;D)$.

On the other hand,  for any $\mu\in\mathcal E^\xi_{\alpha,f}(\mathbf
A,\mathbf 1;\mathbb R^n)$ we have $\mu^+\in\mathcal
E^+_\alpha(\mathbb R^n)$ by the definition of a signed measure of
finite energy, and hence $\mu^+\in\mathcal E_{g,f}^\xi(A_1,1;D)$ by
(\ref{l3-1}) and (\ref{EG}). Because of (\ref{proj}),
(\ref{eq1-2-hen}) and (\ref{EG}),
\begin{align}\label{hope} G_{\alpha,f}(\mu)&=
\|\mu\|^2_\alpha+2\langle f,\mu^+\rangle=\|\mu^+-\mu^-\|^2_\alpha+2\langle f,\mu^+\rangle\\
\notag{}&\geqslant\|\mu^+-(\mu^+)'\|^2_\alpha+2\langle f,\mu^+\rangle
=\|\mu^+\|^2_g+2\langle f,\mu^+\rangle\\
\notag{}&=G_{g,f}(\mu^+)\geqslant G^\xi_{g,f}(A_1,1;D).\notag\end{align}
As $\mu\in\mathcal E^\xi_{\alpha,f}(\mathbf A,\mathbf 1;\mathbb R^n)$ has been chosen arbitrarily,
this together with (\ref{ineq}) proves~(\ref{equality}).

Let now $\lambda^\xi_{A_1}\in\mathcal E_{g,f}^\xi(A_1,1;D)$ satisfy
(\ref{CF}). In the same manner as in the first paragraph of the
present proof we see that
$\breve\mu:=\lambda^\xi_{A_1}-(\lambda^\xi_{A_1})'\in\mathcal
E^\xi_{\alpha,f}(\mathbf A,\mathbf 1;\mathbb R^n)$. Substituting
$\breve\mu$ in place of $\theta$ into (\ref{again}) and then
combining the relation thus obtained with (\ref{equality}), we see
that in fact $G_{\alpha,f}(\breve\mu)=G^\xi_{\alpha,f}(\mathbf
A,\mathbf 1;\mathbb R^n)$. Hence there exists the (unique) solution
$\lambda^\xi_{\mathbf A}:=\breve\mu$ to Problem~\ref{prR}, and it is
related to $\lambda^\xi_{A_1}$ by means of formula~(\ref{reprrr}).

To complete the proof, assume next that $\lambda^\xi_{\mathbf
A}=\lambda^+-\lambda^-\in\mathcal E^\xi_{\alpha,f}(\mathbf A,\mathbf
1;\mathbb R^n)$ satisfies (\ref{inf2}). Similarly as in the second
paragraph of the present proof, we have $\lambda^+\in\mathcal
E_{g,f}^\xi(A_1,1;D)$. Furthermore, by (\ref{equality}) and
(\ref{hope}), the latter with $\lambda^{\xi}_{\mathbf A}$ in place
of~$\mu$,
\begin{align*}G^\xi_{g,f}(A_1,1;D)&=G_{\alpha,f}(\lambda^\xi_{\mathbf A})\geqslant
\|\lambda^+-(\lambda^+)'\|_\alpha^2+2\langle f,\lambda^+\rangle\\
&{}=\|\lambda^+\|^2_g+2\langle f,\lambda^+\rangle=G_{g,f}(\lambda^+)\geqslant G^\xi_{g,f}(A_1,1;D).
\end{align*}
Hence, all the inequalities in the last display are, in fact,
equalities. This shows that $\lambda^\xi_{A_1}:=\lambda^+$ solves
Problem~\ref{prG} and also, on account of (\ref{proj}), that
$\lambda^-=(\lambda^+)'=(\lambda^\xi_{A_1})'$.\end{proof}

When investigating Problem \ref{prG} we shall need the following assertion, see \cite[Lemma~4.3]{DFHSZ}.

\begin{lemma}\label{lequiv} Assume that\/ {\rm(\ref{Gfin})} holds. Then\/
$\lambda\in\mathcal E_{g,f}^\xi(A_1,1;D)$ solves
Problem\/~{\rm\ref{prG}} if and only if
\[\bigl\langle W_{g,f}^\lambda,\nu-\lambda\bigr\rangle\geqslant0\text{ \ for all \ }
\nu\in\mathcal E_{g,f}^\xi(A_1,1;D),\] where it is denoted\/
$W_{g,f}^\lambda:=W_{g,f|_D}^\lambda:=U^\lambda_g+f|_D$, cf.\/
{\rm(\ref{wp})}.
\end{lemma}

\section{Main results}

We keep all the (permanent) assumptions on $\mathbf A$, $f$ and $\xi$ imposed in Sections~\ref{sec:gen1}
and~\ref{sec-def}.

\subsection{Formulations of the main results}\label{form:main} In the following Theorem~\ref{th-suff}
we require that relation (\ref{cggen}) holds; see Lemma~\ref{suff-fin} providing sufficient conditions
for this to occur.

\begin{theorem}\label{th-suff} Suppose moreover that the constraint\/ $\xi\in\mathfrak C(A_1;\mathbb R^n)$
is bounded, i.e.
\begin{equation}\label{boundd}\xi(A_1)<\infty.\end{equation}
Then in either Case\/ {\rm I} or Case\/ {\rm II} Problem\/ {\rm\ref{prR}} is\/ {\rm(}uniquely\/{\rm)}
solvable.
\end{theorem}

Theorem~\ref{th-suff} is sharp in the sense that it does not remain valid if requirement (\ref{boundd})
is omitted from its hypotheses (see the following Theorem~\ref{th-unsuff}).

\begin{theorem}\label{th-unsuff} Condition\/ {\rm(\ref{boundd})} is actually necessary\/
{\rm(}and sufficient\/{\rm)} for the solvability of
Problem\/~{\rm\ref{prR}}. More precisely, suppose that\/
$c_\alpha(A_1)=\infty$ and that Case\/~{\rm II} holds with\/
$\zeta\geqslant0$. Then there exists a constraint\/ $\xi\in\mathfrak
C(A_1;\mathbb R^n)$ with\/ $\xi(A_1)=\infty$ such that\/
$G_{\alpha,f}^{\xi}(\mathbf A,\mathbf 1;\mathbb R^n)=0$, and hence\/
$G_{\alpha,f}^{\xi}(\mathbf A,\mathbf 1;\mathbb R^n)$ cannot be an
actual minimum.
\end{theorem}

The following three assertions establish descriptions of the
$f$-weighted $\alpha$-Riesz potential $W^{\lambda^\xi_{\mathbf
A}}_{\alpha,f}$, cf.\ (\ref{wp}), of the solution
$\lambda^\xi_{\mathbf A}$ to Problem~\ref{prR} (whenever it exists)
and single out its characteristic properties. An analysis of the
support of $\lambda^\xi_{\mathbf A}$ is also provided.

\begin{theorem}\label{desc-pot}Let assumption\/ {\rm(\ref{acirc})} hold and let\/ $f$ be lower bounded
on\/ $A_1$. Fix an arbitrary\/ $\lambda\in\mathcal
E^{\xi}_{\alpha,f}(\mathbf A,\mathbf 1;\mathbb R^n)$; such\/
$\lambda$ exists by Lemma\/~{\rm\ref{suff-fin}}. Then in either
Case\/~{\rm I} or Case\/~{\rm II} the following two assertions are
equivalent:\footnote{In Case~I the assumption of the lower
boundedness of $f$ on $A_1$ is automatically fulfilled. Furthermore,
in Case~I relation (\ref{b2}) is equivalent to the following
apparently stronger assertion: $W^\lambda_{\alpha,f}\leqslant c$ on
$S_D^{\lambda^+}$.}
\begin{itemize}
\item[{\rm(i)}] $\lambda$ is a solution to Problem\/ {\rm\ref{prR}}.
\item[{\rm(ii)}] There exists a number\/ $c\in\mathbb R$ possessing the properties
\begin{align}\label{b1}W^\lambda_{\alpha,f}&\geqslant c\quad(\xi-\lambda^+)\text{-a.e.},\\
\label{b2}W^\lambda_{\alpha,f}&\leqslant c\quad\lambda^+\text{-a.e.},
\end{align}
and in addition it holds true that
\begin{equation}\label{reprrr'}
W^\lambda_{\alpha,f}=0\text{ \ n.e. on\ }A_2.
\end{equation}
If moreover Case\/ {\rm II} holds, then relation\/ {\rm(\ref{reprrr'})} can be rewritten equivalently in
the following apparently stronger form:
\begin{equation}\label{reprrr1'}
W^\lambda_{\alpha,f}=0\text{ \ on\ }A_2\setminus I_{\alpha,A_2},
\end{equation}
where\/ $I_{\alpha,A_2}$ denotes the set of all\/ $\alpha$-irregular\/ {\rm(}boundary\/{\rm)} points
of\/ $A_2$.
\end{itemize}
\end{theorem}

Let $\breve{A}_2$ denote the {\it $\kappa_\alpha$-reduced kernel\/}
of $A_2$ \cite[p.~164]{L}, namely the set of all $x\in A_2$ such
that $c_\alpha\bigl(B(x,r)\cap A_2\bigr)>0$ for every $r>0$.

In the following Theorems~\ref{desc-sup} and \ref{zone} we suppose that there exists the solution
$\lambda_{\mathbf A}^\xi=\lambda^+-\lambda^-$ to Problem~\ref{prR}. For the sake of simplicity of formulation,
in Theorem~\ref{desc-sup} we also assume that in the case $\alpha=2$ the domain $D$ is simply connected.

\begin{theorem}\label{desc-sup}It holds that
\begin{equation}\label{lemma-desc-riesz}
S^{\lambda^-}_{\mathbb R^n}=\left\{
\begin{array}{lll} \breve{A}_2 & \text{if} & \alpha<2,\\ \partial D  & \text{if} & \alpha=2.\\
\end{array} \right.
\end{equation}
\end{theorem}

\begin{theorem}\label{zone}Let\/ $f=0$. Then
\begin{equation}\label{th}W^{\lambda_{\mathbf A}^\xi}_{\alpha,f}=U^{\lambda_{\mathbf A}^\xi}_\alpha=\left\{
\begin{array}{lll} U_g^{\lambda^+} & \text{n.e.\ on} & D,\\
0 & \text{on} & D^c\setminus I_{\alpha,D^c}.\\ \end{array} \right.
\end{equation}
Furthermore, assertion\/ {\rm(ii)} of Theorem\/~{\rm\ref{desc-pot}} holds, and relations\/ {\rm(\ref{b1})}
and\/ {\rm(\ref{b2})} now take respectively the following\/ {\rm(}equivalent\/{\rm)} form:
\begin{align}\label{b21}U^{\lambda_{\mathbf A}^\xi}_\alpha&=c\text{ \ $(\xi-\lambda^+)$-a.e.},\\
\label{b11}U^{\lambda_{\mathbf A}^\xi}_\alpha&\leqslant c\text{ \ on \ }\mathbb R^n,
\end{align}
where\/ $0<c<\infty$. In addition, in the present case\/ $f=0$
relations\/ {\rm(\ref{b21})} and\/ {\rm(\ref{b11})} together with\/
$U^{\lambda_{\mathbf A}^\xi}_\alpha=0$ n.e.\ on\/ $D^c$ determine
uniquely the solution\/ $\lambda_{\mathbf A}^\xi$ to
Problem\/~{\rm\ref{prR}} within the class\/ $\mathcal
E^\xi_{\alpha,f}(\mathbf A,\mathbf 1;\mathbb R^n)$ of admissible
measures. If moreover\/ $U_\alpha^\xi$ is\/ {\rm(}finitely\/{\rm)}
continuous on\/ $D$, then also
\begin{equation}\label{b21'}U^{\lambda_{\mathbf A}^\xi}_\alpha=c\text{ \ on \ }S^{\xi-\lambda^+}_D,
\end{equation}
\begin{equation}\label{cg}c_{g_D^\alpha}\bigl(S_D^{\xi-\lambda^+}\bigr)<\infty.\end{equation}
Omitting now the requirement of the continuity of\/ $U_\alpha^\xi$,
assume next that\/ $\alpha<2$ and\/ $m_n(D^c)>0$, where\/ $m_n$ is
the\/ $n$-dimen\-sional Lebesgue measure. Then
\begin{equation}\label{pg}S_D^{\lambda^+}=S_D^{\xi},\end{equation}
\begin{equation}\label{supp}U^{\lambda_{\mathbf A}^\xi}_\alpha<c\text{ \ on \ }
D\setminus S_D^{\xi}\quad\Bigl({}=D\setminus
S_D^{\lambda^+}\Bigr).\end{equation}
\end{theorem}

The proofs of Theorems~\ref{th-suff}--\ref{zone} are presented in
Section~\ref{sec-proofsss}.

\subsection{An extension of the theory}\label{entire} Parallel with a constraint
$\xi\in\mathfrak C(A_1;\mathbb R^n)$ given by relation (\ref{constr1}) and acting only on (positive)
measures from $\mathcal E^+_\alpha(A_1,1;\mathbb R^n)$, consider also the measure
$\sigma=\sigma^+-\sigma^-\in\mathfrak M(\mathbf A;\mathbb R^n)$ defined as follows:
\begin{equation}\label{sigma}\sigma^+=\xi,\text{ \ while \ }\sigma^-\geqslant\xi'.\end{equation}
Since $\sigma^-(\mathbb R^n)\geqslant\xi'(\mathbb R^n)=\xi(\mathbb
R^n)>1$, where the equality is obtained from
Theorem~\ref{bal-mass-th}, this $\sigma$ can be thought of as a {\it
signed constraint\/} acting on (signed) measures from $\mathcal
E_\alpha(\mathbf A;\mathbf 1;\mathbb R^n)$. Let $\mathcal
E_\alpha^\sigma(\mathbf A,\mathbf 1;\mathbb R^n)$ consist of all
$\mu\in\mathcal E_\alpha(\mathbf A,\mathbf 1;\mathbb R^n)$ such that
$\mu^\pm\leqslant\sigma^\pm$, and let
\begin{equation}\label{sc}G_{\alpha,f}^{\sigma}(\mathbf A,\mathbf 1;\mathbb R^n):=
\inf_{\mu\in\mathcal E_{\alpha,f}^\sigma(\mathbf A,\mathbf 1;\mathbb R^n)}\,G_{\alpha,f}(\mu),\end{equation}
where $\mathcal E_{\alpha,f}^\sigma(\mathbf A,\mathbf 1;\mathbb R^n):=
\mathcal E_\alpha^\sigma(\mathbf A,\mathbf 1;\mathbb R^n)\cap\mathcal E_{\alpha,f}(\mathbb R^n)$.

\begin{theorem}\label{l:eq}With these assumptions and notations, we have
\begin{equation}\label{eq:eq}G_{\alpha,f}^{\sigma}(\mathbf A,\mathbf 1;\mathbb R^n)=
G_{\alpha,f}^{\xi}(\mathbf A,\mathbf 1;\mathbb R^n).
\end{equation}
If these\/ {\rm(}equal\/{\rm)} extremal values are finite, then
Problem\/~{\rm\ref{prR}} {\rm(}with the positive constraint\/
$\xi${\rm)} is solvable if and only so is problem\/ {\rm(\ref{sc})}
{\rm(}with the signed constraint\/ $\sigma${\rm)}, and in the
affirmative case their solutions coincide.
\end{theorem}

\begin{proof} Indeed, $G_{\alpha,f}^{\sigma}(\mathbf A,\mathbf 1;\mathbb R^n)
\geqslant G_{\alpha,f}^{\xi}(\mathbf A,\mathbf 1;\mathbb R^n)$ follows directly from the relation
\begin{equation}\label{pr:inclus}\mathcal E_{\alpha,f}^{\sigma}(\mathbf A,\mathbf 1;\mathbb R^n)
\subset\mathcal E_{\alpha,f}^{\xi}(\mathbf A,\mathbf 1;\mathbb
R^n).\end{equation} To prove the converse inequality, assume
$G_{\alpha,f}^{\xi}(\mathbf A,\mathbf 1;\mathbb R^n)<\infty$ and fix
$\nu\in\mathcal E_{\alpha,f}^{\xi}(\mathbf A,\mathbf 1;\mathbb
R^n)$. Define $\mu:=\nu^+-(\nu^+)'$. It is obvious that
$\mu\in\mathcal E_\alpha(\mathbf A;\mathbb R^n)$, while
Theorem~\ref{bal-mass-th} shows that $(\nu^+)'(A_2)=\nu^+(A_1)=1$.
Furthermore, $(\nu^+)'\leqslant\xi'\leqslant\sigma^-$ by the
linearity of balayage and (\ref{sigma}), and so altogether
$\mu\in\mathcal E_\alpha^{\sigma}(\mathbf A,\mathbf 1;\mathbb R^n)$.
According to (\ref{proj}) and (\ref{EG}), we thus have
\begin{align*}G_{\alpha,f}(\nu)&=\|\nu\|^2_\alpha+2\langle f,\nu^+\rangle
\geqslant\|\nu^+-(\nu^+)'\|^2_\alpha+2\langle f,\nu^+\rangle\\
{}&=\|\mu\|^2_\alpha+2\langle f,\mu^+\rangle=G_{\alpha,f}(\mu)
\geqslant G_{\alpha,f}^\sigma(\mathbf A,\mathbf 1;\mathbb
R^n),\notag\end{align*} which establishes (\ref{eq:eq}) by letting
here $\nu$ range over $\mathcal E_{\alpha,f}^{\xi}(\mathbf A,\mathbf
1;\mathbb R^n)$.

Assume now that (\ref{cggen}) holds. If there is a solution
$\lambda^\sigma_{\mathbf A}$ to problem (\ref{sc}), then this
$\lambda^\sigma_{\mathbf A}$ also solves Problem~\ref{prR}, which is
clear from (\ref{eq:eq}) and (\ref{pr:inclus}). Conversely, if
$\lambda^\xi_{\mathbf A}=\lambda^+-\lambda^-$ solves
Problem~\ref{prR}, then by (\ref{reprrr}) it holds that
$\lambda^-=(\lambda^+)'$, and in the same manner as in the preceding
paragraph we get $\lambda^\xi_{\mathbf A}\in \mathcal
E_\alpha^{\sigma}(\mathbf A,\mathbf 1;\mathbb R^n)$. Hence,
$\lambda^\xi_{\mathbf A}$ also solves problem (\ref{sc}) because
$G_{\alpha,f}(\lambda^\xi_{\mathbf A})= G_{\alpha,f}^{\xi}(\mathbf
A,\mathbf 1;\mathbb R^n)=G_{\alpha,f}^{\sigma}(\mathbf A,\mathbf
1;\mathbb R^n)$ by~(\ref{eq:eq}).\end{proof}

Thus the theory of weighted minimum $\alpha$-Riesz energy problems
with a (positive) constraint $\xi\in\mathfrak C(A_1;\mathbb R^n)$
acting only on positive parts of measures from $\mathcal
E_\alpha(\mathbf A,\mathbf 1;\mathbb R^n)$, developed in
Section~\ref{form:main}, remains valid in its full generality for
the signed constraint $\sigma$, defined by~(\ref{sigma}) and acting
simultaneously on positive and negative parts of $\mu\in\mathcal
E_\alpha(\mathbf A,\mathbf 1;\mathbb R^n)$.

\begin{remark}\label{J}Assume for a moment that a generalized condenser is a finite
collection $\mathbf K=(K_i)_{i\in I}$ of compact sets
$K_i\subset\mathbb R^n$, $i\in I$, with the sign $s_i=\pm1$
prescribed such that
\begin{equation}\label{bzero}c_\alpha(K_i\cap K_j)=0\text{ \ whenever \ }s_is_j=-1.\end{equation}
Problem~\ref{prR}, formulated for $\mathbf K$ in place of $\mathbf
A$, has been analyzed in our recent work \cite{DFHSZ2} for the
$\alpha$-Riesz kernel of {\it any\/} order $\alpha\in(0,n)$, {\it
any\/} normalizing vector $\mathbf a=(a_i)_{i\in I}$, a
vector-valued external field $\mathbf f=(f_i)_{i\in I}$, and a
vector constraint $(\xi^i)_{i\in I}$ such that $U_\alpha^{\xi^i}$ is
(finitely) continuous on $K_i$; see e.g.\ Theorem~6.1 therein.
(Compare with \cite{BC} where a similar problem with $I=\{1,2\}$ and
$\mathbf f=\mathbf 0$ was treated for the logarithmic kernel on the
plane.) However, the approach developed in \cite{DFHSZ2} was based
substantially on the requirement (\ref{bzero}), and can not be
adapted to the present case where $A_2\cap\mathrm{Cl}_{\mathbb
R^n}A_1$ may have nonzero $\alpha$-Riesz capacity.
\end{remark}

\section{Proofs of the assertions formulated in Section \ref{form:main}}\label{sec-proofsss}

Observe that if Case II takes place, then
\begin{equation}\label{IIgg}\zeta\in\mathcal E_g(D),\end{equation}
\begin{equation}\label{IIg}f=U_\alpha^{\zeta-\zeta'}=U_g^\zeta\text{ \ $c_g$-n.e.\ on\ }D.\end{equation}
Indeed, (\ref{IIgg}) is obvious by (\ref{l3-1}), and (\ref{IIg})
holds by Lemma~\ref{l-hatg} and footnote~\ref{RG}. By (\ref{IIgg})
and (\ref{IIg}) we get in Case~II for every $\nu\in\mathcal
E^+_g(A_1;D)$
\begin{equation}\label{est:bel}G_{g,f}(\nu)=\|\nu\|^2_g+2E_g(\zeta,\nu)=\|\nu+\zeta\|^2_g-\|\zeta\|_g^2.
\end{equation}

\subsection{Proof of Theorem \ref{th-suff}} By Theorem~\ref{th:rel}, Theorem~\ref{th-suff} will be proved
once we have established the following assertion.

\begin{theorem}\label{exists}Under the assumptions of Theorem\/~{\rm\ref{th-suff}},
Problem\/~{\rm\ref{prG}} is solvable.\end{theorem}

\begin{proof} Under the assumptions of Theorem~\ref{th-suff} Problem~\ref{prG} makes sense since,
by (\ref{equality}), (\ref{Gfin}) holds. (The value
$G_{g,f}^{\xi}(A_1,1;D)$ is then actually finite, which is clear
from (\ref{equality}) and footnote~\ref{foot-G-finite}.) In view of
(\ref{Gfin}), there is a sequence $\{\mu_k\}_{k\in\mathbb N}\subset
\mathcal E_{g,f}^{\xi}(A_1,1;D)$ such that
\begin{equation}\label{min-seq}\lim_{k\to\infty}\,G_{g,f}(\mu_k)=G_{g,f}^{\xi}(A_1,1;D).\end{equation}
Since $\mathcal E_{g,f}^{\xi}(A_1,1;D)$ is a convex cone and
$\mathcal E_g(D)$ is a pre-Hilbert space with the inner product
$E_g(\nu,\nu_1)$ and the energy norm $\|\nu\|_g=\sqrt{E_g(\nu)}$,
arguments similar to those in the proof of Lemma~\ref{l:uniq} can be
applied to the set $\{\mu_k:\ k\in\mathbb N\}$. This gives
\begin{equation*}0\leqslant\|\mu_k-\mu_\ell\|^2_g\leqslant-
4G^{\xi}_{g,f}(A_1,1;D)+2G_{g,f}(\mu_k)+2G_{g,f}(\mu_\ell).\end{equation*}
Letting here $k,\ell\to\infty$ and then combining the relation thus
obtained with (\ref{min-seq}), we see in view of the finiteness of
$G_{g,f}^{\xi}(A_1,1;D)$ that $\{\mu_k\}_{k\in\mathbb N}$ forms a
strong Cauchy sequence in the metric space $\mathcal E^+_g(D)$. In
particular, this implies
\begin{equation}\label{M}\sup_{k\in\mathbb N}\,\|\mu_k\|_g<\infty.\end{equation}

Since $A_1$ is (relatively) closed in $D$ and the cone $\mathfrak
M^+(D)$ is vaguely closed in $\mathfrak M(D)$, so is the cone
$\mathfrak M^{\xi}(A_1;D):=\{\nu\in\mathfrak M^+(A_1;D):\
\nu\leqslant\xi\}$. Furthermore, the set $\mathfrak
M^\xi(A_1,1;D):=\mathfrak M^{\xi}(A_1;D)\cap\mathfrak M^+(A_1,1;D)$
is vaguely bounded, and hence it is vaguely relatively compact
according to \cite[Chapter~III, Section~2, Proposition~9]{B2}. Thus,
there exists a vague cluster point $\mu$ of the sequence
$\{\mu_k\}_{k\in\mathbb N}$ chosen above, and this $\mu$ belongs to
$\mathfrak M^\xi(A_1;D)$. Passing to a subsequence and changing
notations, we can certainly assume that
\begin{equation}\label{vag-conv}\mu_k\to\mu\text{ \ vaguely in $\mathfrak M^+(D)$ as\ }k\to\infty.
\end{equation}
We assert that this $\mu$ is a solution to Problem~\ref{prG}.

Applying Lemma~\ref{lemma-semi} to $1_D\in\Psi(D)$, we obtain from
(\ref{vag-conv})
\begin{equation*}\label{eq}\mu(A_1)=\mu(D)\leqslant\lim_{k\to\infty}\,\mu_k(D)=1.\end{equation*}
We proceed by showing that equality prevails in the inequality here,
and so altogether
\begin{equation}\label{solution1}\mu\in\mathfrak M^\xi(A_1,1;D).\end{equation}
Consider an exhaustion of $A_1$ by an increasing sequence
$\{K_j\}_{j\in\mathbb N}$ of compact sets. Since $1_{K_j}$ is upper
semicontinuous on $D$ (and of course bounded), we get from
(\ref{vag-conv}) and Lemma~\ref{lemma-semi} with $X=D$ and
$\psi=-1_{K_j}$
\[1\geqslant\mu(A_1)=\lim_{j\to\infty}\,
\mu(K_j)\geqslant\lim_{j\to\infty}\,\limsup_{k\to\infty}\,\mu_k(K_j)
=1-\lim_{j\to\infty}\,\liminf_{k\to\infty}\,\mu_k(A_1\setminus
K_j).\] Relation (\ref{solution1}) will therefore follow if we show
that
\begin{equation}\label{g0}\lim_{j\to\infty}\,\liminf_{k\to\infty}\,\mu_k(A_1\setminus
K_j)=0.\end{equation} Since by (\ref{boundd})
\[\infty>\xi(A_1)=\lim_{j\to\infty}\,\xi(K_j),\]
we have
\[\lim_{j\to\infty}\,\xi(A_1\setminus K_j)=0.\]
When combined with
\[\mu_k(A_1\setminus K_j)\leqslant\xi(A_1\setminus K_j)\text{ \ for any\ }k,j\in\mathbb N,\]
this implies (\ref{g0}) and consequently
(\ref{solution1}).

Another consequence of (\ref{vag-conv}) is that
$\mu_k\otimes\mu_k\to\mu\otimes\mu$ vaguely in $\mathfrak
M^+(D\times D)$ \cite[Chapter~III, Section~5, Exercise~5]{B2}.
Applying Lemma~\ref{lemma-semi} to $X=D\times D$ and $\psi=g$, we
thus get
\[E_g(\mu)\leqslant\liminf_{k\to\infty}\,\|\mu_k\|^2_g<\infty,\]
where the latter inequality is valid by (\ref{M}). Hence,
$\mu\in\mathcal E^+_g(D)$. Combined with (\ref{solution1}), this
yields $\mu\in\mathcal E_g^{\xi}(A_1,1;D)$. As
$G_{f,g}(\mu)>-\infty$, the assertion that $\mu$ solves
Problem~\ref{prG} will therefore be established once we have shown
that
\begin{equation}\label{incl3}G_{g,f}(\mu)\leqslant\lim_{k\to\infty}\,G_{g,f}(\mu_k).\end{equation}

Since the kernel $g$ is perfect \cite[Theorem~4.11]{FZ}, the sequence $\{\mu_k\}_{k\in\mathbb N}$,
being strong Cauchy in $\mathcal E^+_g(D)$ and vaguely convergent to $\mu$, converges to the same limit
strongly in $\mathcal E_g^+(D)$, i.e.
\begin{equation}\label{conv-str}\lim_{k\to\infty}\,\|\mu_k-\mu\|_g=0.\end{equation}
Also note that the mapping $\nu\mapsto G_{g,f}(\nu)$ is vaguely
l.s.c., resp.\ strongly continuous, on $\mathcal
E_{g,f}(D)\cap\mathfrak M^+(A_1;D)$ if Case~I, resp.\ Case~II,
holds. In fact, since $\|\nu\|_g$ is vaguely l.s.c.\ on $\mathcal
E^+_g(D)$, the former assertion follows from Lemma~\ref{lemma-semi}.
As for the latter assertion, it is obvious by (\ref{est:bel}). This
observation enables us to obtain (\ref{incl3}) from (\ref{vag-conv})
and~(\ref{conv-str}).\end{proof}

\subsection{Proof of Theorem~\ref{th-unsuff}} In view of Theorem~\ref{th-suff}, it is enough to
establish the necessity part of the theorem. Assume that the
requirements of the latter part of the theorem are fulfilled. Since
Case II with $\zeta\geqslant0$ takes place, we get from (\ref{IIgg})
and~(\ref{IIg})
\begin{equation}\label{caseII}G_{g,f}(\nu)=\|\nu\|_g^2+2E_g(\zeta,\nu)\in[0,\infty)\text{ \ for all\ }
\nu\in\mathcal E^+_g(A_1;D).\end{equation}

Consider numbers $r_j>0$, $j\in\mathbb N$, such that
$r_j\uparrow\infty$ as $j\to\infty$, and write $B_{r_j}:=B(0,r_j)$,
$A_{1,r_j}:=A_1\cap B_{r_j}$. As $c_\alpha(B_{r_j})<\infty$ and
$c_\alpha(A_1)=\infty$, it follows from the subadditivity of
$c_\alpha(\cdot)$ on universally measurable sets
\cite[Lemma~2.3.5]{F1} that $c_\alpha(A_1\setminus B_{r_j})=\infty$.
For every $j\in\mathbb N$ there is therefore $\xi_j\in\mathcal
E_\alpha^+(A_1\setminus B_{r_j},1;\mathbb R^n)$ of compact support
$S_D^{\xi_j}$ such that
\begin{equation}\label{to0}\|\xi_j\|_\alpha\leqslant j^{-2}.\end{equation}
Clearly, the $r_j$ can be chosen successively so that $A_{1,r_j}\cup
S^{\xi_j}_D\subset A_{1,r_{j+1}}$. Any compact set $K\subset\mathbb
R^n$ is contained in a ball $B_{r_{j_0}}$ with $j_0$ large enough,
and hence $K$ has points in common with only finitely many
$S^{\xi_j}_D$. Therefore $\xi$ defined by the relation
\begin{equation}\label{xixi}\xi(\varphi):=\sum_{j\in\mathbb N}\,\xi_j(\varphi)\text{ \ for
any\ }\varphi\in C_0(\mathbb R^n)\end{equation} is a positive Radon
measure on $\mathbb R^n$ carried by $A_1$. Furthermore,
$\xi(A_1)=\infty$ and $\xi\in\mathcal E^+_\alpha(\mathbb R^n)$. To
prove the latter, note that $\eta_k:=\xi_1+\dots+\xi_k\in\mathcal
E^+_\alpha(\mathbb R^n)$ in view of (\ref{to0}) and the triangle
inequality in $\mathcal E_\alpha(\mathbb R^n)$. Also observe that
$\eta_k\to\xi$ vaguely in $\mathfrak M(\mathbb R^n)$ because for any
$\varphi\in C_0(\mathbb R^n)$ there is $k_0$ such that
$\xi(\varphi)=\eta_k(\varphi)$ for all $k\geqslant k_0$. As
$\|\eta_k\|_\alpha\leqslant M:= \sum_{j\in\mathbb N}\,j^{-2}<\infty$
for all $k\in\mathbb N$, Lemma~\ref{lemma-semi} with $X=\mathbb
R^n\times\mathbb R^n$ and $\psi=\kappa_\alpha$ yields
$\|\xi\|_\alpha\leqslant M$. It has thus been shown that $\xi$ given
by (\ref{xixi}) is an element of $\mathfrak C(A_1;\mathbb R^n)$ with
$\xi(A_1)=\infty$.

Each $\xi_j$ belongs to $\mathcal E^+_g(A_1,1;D)$ and moreover, by
(\ref{g-ineq}) and~(\ref{to0}),
\begin{equation}\label{estell}\|\xi_j\|_g\leqslant\|\xi_j\|_\alpha\leqslant j^{-2}.\end{equation}
Since Case II takes place, $\xi_j\in\mathcal E_{g,f}^{\xi}(A_1,1;D)$
for all $j\in\mathbb N$ by (\ref{caseII}). By the Cau\-chy--Schwarz
(Bunyakovski) inequality in the pre-Hilbert space $\mathcal
E_{g}(D)$,
\[0\leqslant G_{g,f}^{\xi}(A_1,1;D)\leqslant\lim_{j\to\infty}\,\bigl[\|\xi_j\|_g^2+
2E_g(\zeta,\xi_j)\bigr]\leqslant2\|\zeta\|_g\lim_{j\to\infty}\,\|\xi_j\|_g=0,\]
where the first and the second inequalities hold by (\ref{caseII}),
while the third inequality and the equality are valid by
(\ref{estell}). Hence, $G_{g,f}^{\xi}(A_1,1;D)=0$. As seen from
(\ref{caseII}), such infimum can be attained only at zero measure,
which is impossible because $0\notin\mathcal
E_{g,f}^{\xi}(A_1,1;D)$. Combined with Theorem~\ref{th:rel}, this
establishes the claimed assertion.

\subsection{Proof of Theorem~\ref{desc-pot}}
Fix $\lambda=\lambda^+-\lambda^-\in\mathcal E^\xi_{\alpha,f}(\mathbf
A,\mathbf 1;\mathbb R^n)$, and note that since $f=0$ n.e.\ on $A_2$,
relation (\ref{reprrr'}) can alternatively be rewritten as
$U_\alpha^\lambda=U_\alpha^{\lambda^+-\lambda^-}=0$ n.e.\ on $A_2$,
which in view of the $c_\alpha$-abs\-olute continuity of $\lambda$
and (\ref{bal-eq}) is equivalent to the equality
\begin{equation}\label{lbal}\lambda^-=(\lambda^+)'.\end{equation}
Taking Theorem~\ref{th:rel} into account, we thus see that, while proving the equivalence of assertions
(i) and (ii) of Theorem~\ref{desc-pot}, there is no loss of generality in assuming that the given measure
$\lambda$ satisfies (\ref{lbal}). By (\ref{hatg}) we therefore get
\[U_\alpha^{\lambda}=U_\alpha^{\lambda^+-(\lambda^+)'}=
\left\{
\begin{array}{ll} U_g^{\lambda^+} & \mbox{n.e.\ on\ } D,\\ 0 & \mbox{n.e.\ on\ }  A_2,\\ \end{array} \right.\]
and hence
\[W_{\alpha,f}^\lambda=\left\{
\begin{array}{ll} W_{g,f}^{\lambda^+} & \mbox{n.e.\ on\ } D,\\ 0 & \mbox{n.e.\ on\ } A_2,\\ \end{array}
\right.\]
where $W_{g,f}^{\lambda^+}:=W_{g,f|_D}^{\lambda^+}=U_g^{\lambda^+}+f|_D$, cf.\ (\ref{wp}).
If moreover Case~II holds, then
\[W_{\alpha,f}^{\lambda}=U_\alpha^{\lambda^++\zeta}-U_\alpha^{(\lambda^++\zeta)'}
\text{ \ n.e.\ on\ }\mathbb R^n.\] According to
\cite[Corollary~3.14]{FZ}, the function on the right (hence that on
the left) in this relation takes the value $0$ at every
$\alpha$-regular point of $A_2$, which establishes~(\ref{reprrr1'}).

Combined with Theorem~\ref{th:rel}, what has been shown just above
yields that Theorem~\ref{desc-pot} will be proved once we have
established the following theorem.

\begin{theorem}\label{desc-pot-g} Under the hypotheses of Theorem\/~{\rm\ref{desc-pot}} the following two
assertions are equivalent for any\/ $\mu\in\mathcal E^\xi_{g,f}(A_1,1;D)$:
\begin{itemize}
\item[{\rm(i$'$)}] $\mu$ is a solution to Problem\/ {\rm\ref{prG}}.
\item[{\rm(ii$'$)}] There exists a number\/ $c\in\mathbb R$ possessing the properties
\begin{align}\label{b1'}W^{\mu}_{g,f}&\geqslant c\quad(\xi-\mu)\text{-a.e.},\\
\label{b2'}W^{\mu}_{g,f}&\leqslant c\quad\mu\text{-a.e.}
\end{align}
\end{itemize}
\end{theorem}

\begin{proof} Throughout the proof we shall use permanently the fact that both $\xi$ and $\mu$ have
finite $\alpha$-Riesz energy, are hence they are
$c_\alpha$-absolutely continuous.

Suppose first that assertion (i$'$) holds. Inequality (\ref{b1'}) is
valid for $c=L$, where
\[L:=\sup\,\bigl\{q\in\mathbb R: \ W_{g,f}^\mu\geqslant q\quad(\xi-\mu)\mbox{-a.e.}\bigr\}.\] In turn,
(\ref{b1'}) with $c=L$ implies that $L<\infty$ because $W_{g,f}^{\mu}<\infty$ holds n.e.\ on $A_1^\circ$
and hence $(\xi-\mu)$-a.e.\ on $A_1^\circ$, while
$(\xi-\mu)(A_1^\circ)>0$ by (\ref{acirc}). Also note that $L>-\infty$, for $W_{g,f}^{\mu}$ is
lower bounded on $A_1$ by assumption.

We next proceed by establishing (\ref{b2'}) with $c=L$. To this end
write for any $w\in\mathbb R$
\[A_1^+(w):=\bigl\{x\in A_1:\ W_{g,f}^\mu(x)>w\bigr\}\quad\text{and}\quad
A_1^-(w):=\bigl\{x\in A_1:\ W_{g,f}^\mu(x)<w\bigr\}.\] Assume on the
contrary that (\ref{b2'}) with $c=L$ fails, i.e.\ $\mu(A_1^+(L))>0$.
Since $W_{g,f}^\mu$ is $\mu$-measurable, one can choose
$w'\in(L,\infty)$ so that $\mu(A_1^+(w'))>0$. At the same time, as
$w'>L$, (\ref{b1'}) with $c=L$  yields $(\xi-\mu)(A_1^-(w'))>0$.
Therefore, there exist compact sets $K_1\subset A_1^+(w')$ and
$K_2\subset A_1^-(w')$ such that
\begin{equation}\label{c01}0<\mu(K_1)<(\xi-\mu)(K_2).\end{equation}

Write $\tau:=(\xi-\mu)|_{K_2}$; then
$E_g(\tau)<E_\alpha(\tau)<\infty$. Since $\langle
W_{g,f}^{\mu},\tau\rangle\leqslant w'\tau(K_2)<\infty$, we thus get
$\langle f,\tau\rangle<\infty$. Define
$\theta:=\mu-\mu|_{K_1}+b\tau$, where
$b:=\mu(K_1)/\tau(K_2)\in(0,1)$ by (\ref{c01}). Straightforward
verification then shows that $\theta(A_1)=1$ and
$\theta\leqslant\xi$, and hence $\theta\in\mathcal
E^{\xi}_{g,f}(A_1,1;D)$. On the other hand,
\begin{align*}
\langle W_{g,f}^\mu,\theta-\mu\rangle&=\langle
W_{g,f}^\mu-w',\theta-\mu\rangle\\&{}=-\langle
W_{g,f}^\mu-w',\mu|_{K_1}\rangle+b\langle
W_{g,f}^\mu-w',\tau\rangle<0,\end{align*} which is impossible in
view of Lemma~\ref{lequiv} applied to $\lambda=\mu$ and
$\nu=\theta$. This contradiction establishes (\ref{b2'}), thus
completing the proof that (i$'$) implies~(ii$'$).

Conversely, let (ii$'$) hold. Then $\mu(A_1^+(c))=0$ and
$(\xi-\mu)(A_1^-(c))=0$. For any $\nu\in\mathcal
E^\xi_{g,f}(A_1,1;D)$ we therefore obtain
\begin{align*}\langle
W_{g,f}^\mu,\nu-\mu\rangle&=\langle
W_{g,f}^\mu-c,\nu-\mu\rangle\\
&{}=\bigl\langle W_{g,f}^\mu-c,\nu|_{A_1^+(c)}\bigr\rangle+\bigl\langle
W_{g,f}^\mu-c,(\nu-\xi)|_{A_1^-(c)}\bigr\rangle\geqslant0.\end{align*}
Application of Lemma~\ref{lequiv} shows that, indeed, $\mu$ is the solution to Problem~\ref{prG}.\end{proof}

\subsection{Proof of Theorem~\ref{desc-sup}} For any $x\in D$ let $K_x$ be the inverse of
$\mathrm{Cl}_{\overline{\mathbb R^n}}\,A_2$ relative to $S(x,1)$,
$\overline{\mathbb R^n}$ being the one-point compactification of
$\mathbb R^n$. Since $K_x$ is compact, there is the (unique)
$\kappa_\alpha$-equil\-ibrium measure $\gamma_x\in\mathcal
E^+_\alpha(K_x;\mathbb R^n)$ on $K_x$ with the properties
$\|\gamma_x\|^2_\alpha=\gamma_x(K_x)=c_\alpha(K_x)$,
\begin{equation}\label{Keq}U_\alpha^{\gamma_x}=1\text{ \ n.e.\ on\ }K_x,\end{equation}
and $U_\alpha^{\gamma_x}\leqslant1$ on $\mathbb R^n$. Note that
$\gamma_x\ne0$, for $c_\alpha(K_x)>0$ in consequence of
$c_\alpha(A_2)>0$, see \cite[Chapter~IV, Section~5,
n$^\circ$\,19]{L}. We assert that under the stated requirements
\begin{equation}\label{eq-desc}
S^{\gamma_x}_{\mathbb R^n}=\left\{
\begin{array}{lll} \breve{K}_x & \text{if} & \alpha<2,\\ \partial_{\mathbb R^n} K_x  &
\text{if} & \alpha=2.\\ \end{array} \right.
\end{equation}
The latter equality in (\ref{eq-desc}) follows from \cite[Chapter~II, Section~3, n$^\circ$\,13]{L}.
To establish the former equality,\footnote{We have brought here this proof, since we did not find a reference
for this possibly known assertion.} we first note that $S^{\gamma_x}_{\mathbb R^n}\subset\breve{K}_x$ by
the $c_\alpha$-absolute continuity of $\gamma_x$. As for the converse inclusion, assume on the contrary that
there is $x_0\in\breve{K}_x$ such that $x_0\notin S^{\gamma_x}_{\mathbb R^n}$. Choose $r>0$ with the
property  $\overline{B}(x_0,r)\cap S^{\gamma_x}_{\mathbb R^n}=\varnothing$. But
$c_\alpha\bigl(B(x_0,r)\cap\breve{K}_x\bigr)>0$, hence there is $y\in B(x_0,r)$ such that
$U_\alpha^{\gamma_x}(y)=1$. The function $U_\alpha^{\gamma_x}$ is $\alpha$-har\-monic on $B(x_0,r)$
\cite[Chapter~I, Section~5, n$^\circ$\,20]{L},
continuous on $\overline{B}(x_0,r)$, and takes at $y\in B(x_0,r)$ its maximum value $1$.
Applying \cite[Theorem~1.28]{L} we obtain $U_\alpha^{\gamma_x}=1$ $m_n$-a.e.\ on $\mathbb R^n$,
hence everywhere on $(\breve{K}_x)^c$ by the continuity of $U_\alpha^{\gamma_x}$ on
$\bigl(S^{\gamma_x}_{\mathbb R^n}\bigr)^c$ \ $\bigl[{}\supset(\breve{K}_x)^c\bigr]$, and altogether n.e.\
on $\mathbb R^n$ by (\ref{Keq}). This means that $\gamma_x$ serves as the $\alpha$-Riesz equilibrium measure
on the whole of $\mathbb R^n$, which is impossible.

Based on (\ref{reprrr}) and the integral representation
(\ref{int-repr}), we then arrive at (\ref{lemma-desc-riesz}) with
the aid of the fact that for every $x\in D$, $\varepsilon_x'$ is the
Kelvin transform of the $\kappa_\alpha$-equilibrium measure
$\gamma_x$, see \cite[Section~3.3]{FZ}.

\subsection{Proof of Theorem~\ref{zone}}\label{proof:zone} Since $\lambda^-=(\lambda^+)'$ by (\ref{reprrr})
and $f=0$ by assumption, the function
\[W_{\alpha,f}^{\lambda^\xi_{\mathbf A}}=U_\alpha^{\lambda^\xi_{\mathbf A}}=U_\alpha^{\lambda^+}-
U_\alpha^{(\lambda^+)'}\]
is well defined and finite n.e.\ on $\mathbb R^n$. In particular, it is well defined on all of $D$ and
it equals there the strictly positive function $U_g^{\lambda^+}$, see Lemma~\ref{l-hatg}. This together with
(\ref{reprrr1'}) proves (\ref{th}). Combining (\ref{th}) with  (\ref{b2}) shows that under the stated
assumptions the number $c$ from Theorem~\ref{desc-pot} is ${}>0$, while (\ref{b1}) now takes the (equivalent)
form
\begin{equation}\label{err}U_\alpha^{\lambda^\xi_{\mathbf A}}\geqslant c>0\quad (\xi-\lambda^+)
\text{-a.e.}\end{equation}
Having rewritten (\ref{b2}) as
\[U_\alpha^{\lambda^+}\leqslant U_\alpha^{\lambda^-}+c\quad\lambda^+\text{-a.e.},\]
we infer from \cite[Theorems~1.27, 1.29, 1.30]{L} that the same
inequality holds on all of $\mathbb R^n$, which amounts to
(\ref{b11}). In turn, (\ref{b11}) yields (\ref{b21}) when combined
with~(\ref{err}). It follows directly from Theorem~\ref{desc-pot}
that relations (\ref{b21}) and (\ref{b11}) together with
$U^{\lambda_{\mathbf A}^\xi}_\alpha=0$ n.e.\ on $D^c$ determine
uniquely the solution $\lambda_{\mathbf A}^\xi$ to Problem~\ref{prR}
within the class $\mathcal E^\xi_{\alpha,f}(\mathbf A,\mathbf
1;\mathbb R^n)$ of admissible measures.

Assume now that $U_\alpha^\xi$ is continuous on $D$. Then so is
$U_\alpha^{\lambda^+}$. Indeed, since $U_\alpha^{\lambda^+}$ is
l.s.c.\ and
$U_\alpha^{\lambda^+}=U_\alpha^\xi-U_\alpha^{\xi-\lambda^+}$ with
$U_\alpha^\xi$ continuous on $D$ and $U_\alpha^{\xi-\lambda^+}$
l.s.c., it follows that $U_\alpha^{\lambda^+}$ is also upper
semicontinuous, and hence continuous. Therefore, by the continuity
of $U_\alpha^{\lambda^+}$ on $D$, (\ref{b21}) implies (\ref{b21'}).
Thus, by (\ref{th}) and (\ref{b21'}),
\[U_g^{\lambda^+}=c\text{ \ on \ }S_D^{\xi-\lambda^+},\]
which implies (\ref{cg}) in view of \cite[Lemma 3.2.2]{F1} with $\kappa=g$.

Omitting now the requirement of the continuity of $U_\alpha^\xi$,
assume next that $\alpha<2$ and $m_n(D^c)>0$. If on the contrary
(\ref{pg}) fails, then there is $x_0\in S_D^\xi$ such that
$x_0\notin S_D^{\lambda^+}$. Thus one can choose $r>0$ so that
\begin{equation}\label{ball}\overline{B}(x_0,r)\subset D\quad\text{and}\quad
\overline{B}(x_0,r)\cap S^{\lambda^+}_D=\varnothing.\end{equation}
Then $(\xi-\lambda^+)\bigl(\overline{B}(x_0,r)\bigr)>0$, and hence
by (\ref{b21}) there exists $y\in\overline{B}(x_0,r)$ with the
property $U_\alpha^{\lambda_{\mathbf A}^\xi}(y)=c$, or equivalently
\begin{equation}\label{eq;eq}U_\alpha^{\lambda^+}(y)=U_\alpha^{\lambda^-}(y)+c.\end{equation}
As $U_\alpha^{\lambda^+}$ is $\alpha$-harmonic on $B(x_0,r)$ and
continuous on $\overline{B}(x_0,r)$, while $U_\alpha^{\lambda^-}+c$
is $\alpha$-super\-harmonic on $\mathbb R^n$, we obtain from
(\ref{b11}) and (\ref{eq;eq}) with the aid of \cite[Theorem~1.28]{L}
\begin{equation}\label{contr}U_\alpha^{\lambda^+}=U_\alpha^{\lambda^-}+c\quad
m_n\mbox{-a.e.\ on \ }\mathbb R^n.\end{equation} This implies $c=0$,
for
$U_\alpha^{\lambda^+}=U_\alpha^{(\lambda^+)'}=U_\alpha^{\lambda^-}$
holds n.e.\ on $D^c$, and hence $m_n$-a.e.\ on $D^c$. A
contradiction.

Similar arguments enable us to establish (\ref{supp}). Indeed, if
(\ref{supp}) fails at some $x_1\in D\setminus S_D^{\lambda^+}$, then
relation (\ref{eq;eq}) would be valid with $x_1$ in place of $y$,
see (\ref{b11}); and moreover one could choose $r>0$ so that
(\ref{ball}) would be fulfilled with $x_1$ in place of $x_0$.
Therefore, using the $\alpha$-har\-mon\-icity of
$U_\alpha^{\lambda^+}$ on $B(x_1,r)$ as well as the
$\alpha$-super\-harmonicity of $U_\alpha^{\lambda^-}+c$ on $\mathbb
R^n$, we would arrive again at (\ref{contr}), and hence at the
equality $c=0$. The contradiction thus obtained completes the proof
of the theorem.

\section{Duality relation between non-weighted constrained and weighted unconstrained minimum $\alpha$-Green
energy problems}\label{sec:dual}
As above, fix a (not necessarily proper) subset $A_1$ of $D$ which is relatively closed in $D$ and fix a
constraint $\xi\in\mathfrak C(A_1;\mathbb R^n)$, see (\ref{constr1}), with $1<\xi(A_1)<\infty$;
such $\xi$ exists because of the (permanent) assumption $c_\alpha(A_1)>0$. According to Theorem~\ref{exists},
the non-weighted ($f=0$) constrained minimum $\alpha$-Green energy problem over the class
$\mathcal E_g^\xi(A_1,1;D)$ is (uniquely) solvable, i.e.\ there exists
$\lambda=\lambda_{A_1}^\xi\in\mathcal E_g^\xi(A_1,1;D)$ with
\begin{equation}\label{eq:constr}\|\lambda\|_g^2=
\min_{\nu\in\mathcal E_g^\xi(A_1,1;D)}\,\|\nu\|^2_g.\end{equation}
Write $q:=[\xi(A_1)-1]^{-1}$ and
\[\theta:=q(\xi-\lambda),\quad f_0:=-qU^\xi_g.\]

\begin{theorem}\label{th:dual}Assume moreover that\/ $U_g^\xi$ is\/ {\rm(}finitely\/{\rm)} continuous
on\/ $D$. Then the measure\/ $\theta$ is a\/ {\rm(}unique\/{\rm)}\/ solution to the\/
$f_0$-weighted unconstrained minimum\/ $\alpha$-Green energy problem over\/ $\mathcal E_g^+(A_1,1;D)$,
i.e. $\theta\in\mathcal E_g^+(A_1,1;D)$ and
\begin{equation}\label{eq:dual}G_{g,f_0}(\theta)=\inf_{\nu\in\mathcal E_g^+(A_1,1;D)}\,G_{g,f_0}(\nu).
\end{equation}
Moreover, there exists\/ $\eta\in(0,\infty)$ such that
\begin{align}
\label{Wsc1}W_{g,f_0}^\theta&=-\eta\text{ \ on \ }S^\theta_D,\\
\label{Wsc2}W_{g,f_0}^\theta&\geqslant-\eta\text{ \ on \ }D,
\end{align}
and these two relations determine uniquely a solution to problem\/
{\rm(\ref{eq:dual})} among the measures of the class\/ $\mathcal
E_g^+(A_1,1;D)$.
\end{theorem}

\begin{proof} Under the stated assumptions, relations (\ref{b1'}) and (\ref{b2'}) for the solution
$\lambda$ to the (non-weighted constrained) problem (\ref{eq:constr}) take the (equivalent) form
\begin{align}\label{pr1}U_g^\lambda&\geqslant c\text{ \ $(\xi-\lambda)$-a.e.,}\\
\label{pr2}U_g^\lambda&\leqslant c\text{ \ $\lambda$-a.e.}\end{align}
Thus $c>0$, see (\ref{pr2}). Applying Theorem~\ref{th-dom-pr} with $v=c$, from (\ref{pr2}) we
therefore obtain
\[U_g^\lambda\leqslant c\text{ \ on\ }D.\]
Combined with (\ref{pr1}), this gives $U_g^\lambda=c$
$(\xi-\lambda)$-a.e., and hence
\begin{equation*}U_g^\lambda=c\text{ \ on\ }S^{\xi-\lambda}_D,\end{equation*}
for $U_g^\lambda$ is (finitely) continuous on $D$ along with $U_g^\xi$. (Indeed, the continuity of
$U_g^\lambda$ follows in the same manner as in Section~\ref{proof:zone}, see the second paragraph,
with $g$ in place of~$\kappa_\alpha$.)

With the chosen notations the two preceding displays can
alternatively be rewritten as (\ref{Wsc1}) and (\ref{Wsc2}) with
$\eta:=qc$. In turn, (\ref{Wsc1}) and (\ref{Wsc2}) imply that
$\theta$, $f_0$ and $-\eta$ satisfy relations (7.9) and (7.10) in
\cite{ZPot2}, which according to \cite[Theorem~7.3]{ZPot2}
establishes (\ref{eq:dual}).\end{proof}

\section{Examples}

The purpose of the examples below is to illustrate the assertions from Section~\ref{form:main}.
Observe that both in Example~\ref{ex-2} and Example~\ref{ex-3} the set $A_2=D^c$ is not $\alpha$-thin at
infinity.

\begin{example}\label{ex-2}Let $n\geqslant3$, $0<\alpha<2$, $A_1=D=B(0,r)$, where $r\in(0,\infty)$, and
let $A_2=D^c$, $f=0$. Define $\xi:=q\lambda_r$, where $q>1$ and
$\lambda_r$ is the $\kappa_\alpha$-capacitary measure on
$\overline{B}(0,r)$, see Remark~\ref{remark}. As follows from
\cite[Chapter~II, Section~3, n$^\circ$\,13]{L}, $\xi\in\mathcal
E_\alpha^+(A_1,q;\mathbb R^n)$, $S_D^{\xi}=D$ and $U_\alpha^\xi$ is
continuous on $\mathbb R^n$. Since $f=0$, Problem~\ref{prR} reduces
to the problem of minimizing $E_\alpha(\mu)$ over the class of all
(signed Radon) measures $\mu\in\mathcal E_\alpha(\mathbf A,\mathbf
1;\mathbb R^n)$ with $\mu^+\leqslant\xi$, which by
Theorem~\ref{th:rel} is equivalent to the problem of minimizing
$E_{g^\alpha_D}(\nu)$ where $\nu$ ranges over $\mathcal
E^\xi_{g^\alpha_D}(A_1,1;D)$. According to Theorems~\ref{th:rel},
\ref{th-suff} and Lemma~\ref{l:uniq}, these two constrained minimum
energy problems are uniquely solvable (no short-circuit occurs) and
their solutions, denoted respectively by $\lambda_\mathbf
A^\xi=\lambda^+-\lambda^-$ and $\lambda^\xi_{A_1}$, are related to
each other as in (\ref{reprrr}). Furthermore, by
(\ref{lemma-desc-riesz}), (\ref{cg}) and (\ref{pg}) we obtain
\[S_D^{\lambda^+}=S_D^{\lambda^\xi_{A_1}}=S_D^\xi=D,\quad S_{\mathbb R^n}^{\lambda^-}=D^c,\]
\begin{equation}\label{hryu}c_{g_D^\alpha}\bigl(S_D^{\xi-\lambda^+}\bigr)<\infty,\end{equation}
and finally by (\ref{reprrr1'}) and (\ref{b21'}) we have
\begin{equation}\label{det}U_\alpha^{\lambda_\mathbf A^\xi}=\left\{
\begin{array}{lll} c & \text{on} & S^{\xi-\lambda^+}_D,\\ 0  & \text{on} & D^c,\\
\end{array} \right.\end{equation}
where $c>0$, while by (\ref{b11})
\begin{equation}\label{ex1}U_\alpha^{\lambda_\mathbf A^\xi}\leqslant c\text{ \ on \ }
D\setminus S^{\xi-\lambda^+}_D.\end{equation}
Moreover, according to Theorem~\ref{desc-pot} relations (\ref{det}) and (\ref{ex1}) determine uniquely
the solution $\lambda_\mathbf A^\xi$ among the class of admissible measures.
\end{example}

\begin{example}\label{ex-3}Let $n=3$, $\alpha=2$, $f=0$ and let
$D:=\bigl\{x=(x_1,x_2,x_3)\in\mathbb R^3: \  x_1>0\bigr\}$. Define $A_1$ as the union of $K_k$ over
$k\in\mathbb N$,
where
\[K_k:=\Bigl\{(x_1,x_2,x_3)\in\mathbb R^3: \ x_1=\frac1k, \ x_2^2+x_3^2\leqslant k^2\Bigr\},
\quad k\in\mathbb N.\]
Let $\lambda_k$ be the $\kappa_2$-capacitary measure on $K_k$, see Remark~\ref{remark};
hence $\lambda_k(K_k)=1$ and $\|\lambda_k\|^2_2=\pi^2/(2k)$ by \cite[Chapter~II, Section~3, n$^\circ$\,14]{L}.
Define
\[\xi:=\sum_{k\in\mathbb N}\,\frac{\lambda_k}{k^2}.\]
In the same manner as in the proof of Theorem~\ref{th-unsuff} one
can see that $\xi$ is a bounded positive Radon measure carried by
$A_1$ with $E_2(\xi)<\infty$. Therefore it follows from
Theorem~\ref{th-suff} that Problem~\ref{prR} for the constraint
$\xi$ and the generalized condenser $\mathbf A=(A_1,D^c)$ has a
solution $\lambda_\mathbf A^\xi$ (no short-circ\-uit occurs),
although $D^c\cap\mathrm{Cl}_{\mathbb R^3}A_1=\partial D=\{x_1=0\}$
and hence
\[c_2\bigl(D^c\cap\mathrm{Cl}_{\mathbb R^3}A_1\bigr)=\infty.\]
Furthermore, since each $U_2^{\lambda_k}$, $k\in\mathbb N$, is
continuous on $\mathbb R^n$ and bounded from above by $\pi^2/(2k)$,
the potential $U_2^\xi$ is continuous on $\mathbb R^n$ by uniform
convergence of the sequence $\sum_{k\in\mathbb
N}\,k^{-2}U_2^{\lambda_k}$. Hence, (\ref{hryu}), (\ref{det}) and
(\ref{ex1}) also hold in the present case with $\alpha=2$, again
with $c>0$, and relations (\ref{det}) and (\ref{ex1}) determine
uniquely the solution $\lambda_\mathbf A^\xi$ within the class of
admissible measures. Also note that $S^{\lambda^-}_{\mathbb
R^n}=\partial D$ according to~(\ref{lemma-desc-riesz}).
\end{example}

\section{Appendix}\label{app}

The following example shows that even for positive bounded (hence extendible) measures on an open ball
in $\mathbb R^3$ the finiteness of the $\alpha$-Green energy does not necessarily imply the finiteness
of the $\alpha$-Riesz energy, contrary to what was stated in \cite[Lemma~2.4]{DFHSZ}.

\begin{example}\label{counterex} Let $\alpha=2$. For technical simplicity we first construct the
analogous example with the ball replaced by the half-space
$D=\{(x_1,x_2,x_3)\in\mathbb R^3:\ x_1>0\}$ (next we apply a Kelvin
transformation). The boundary $\partial D$ (replacing the sphere) is
then the plane $\{x_1=0\}$. For $r>0$ let $\mu_r$ denote the
$\kappa_2$-capacitary measure on the closed $2$-dim\-ens\-ion\-al
disc $K_r\subset\partial D$ of radius $r$ centered at $(0,0,0)$, see
Remark~\ref{remark}. Such $\mu_r$ exists since $0<c_2(K_r)<\infty$
(in fact $c_2(K_r)=2r/\pi^2$, see \cite[Chapter~II, Section~3,
n$^\circ$\,14]{L}). The Newtonian energy $E_2(\mu_r)$ equals
$E_2(\mu_1)/r$, where $0<E_2(\mu_1)=1/c_2(K_1)<\infty$. For real
numbers $z_1$ and $z_2$ and a measure $\nu\in\mathfrak M^+(\partial
D;\mathbb R^3)$ denote by $\nu^{z_1,z_2}$ the translation of $\nu$
in $\mathbb R^3$ by the vector $(0,z_1,z_2)$. Then $\mu_r^{z_1,z_2}$
is the $\kappa_2$-capacitary measure on the translation of the disk
$K_r$ by the vector $(0,z_1,z_2)$, denoted by $K_r^{z_1,z_2}$.

For fixed $r>0$ the potential $U_2^{\mu_r}$ on $\mathbb R^3$ equals
$1$ on the disc $K_r$ by the Wiener criterion. By the continuity
principle \cite[Theorem~1.7]{L}, $U_2^{\mu_r}$ is (finitely)
continuous on $\mathbb R^3$, and even uniformly since
$U_2^{\mu_r}(x)\to0$ uniformly as $|x|\to\infty$, $S_{\mathbb
R^3}^{\mu_r}$ being compact (actually, $S_{\mathbb
R^3}^{\mu_r}=K_r$).

For any positive measure $\nu$ on $\mathbb R^3$ we denote by
$(\nu)\,\check{}$ the image of $\nu$ under the reflection
$(x_1,x_2,x_3)\mapsto(-x_1,x_2,x_3)$ with respect to $\partial D$.
The $2$-Green kernel $g=g^2_D$ on the half-spa\-ce $D$ is given by
\[g(x,y)=U_2^{\varepsilon_y}(x)-U_2^{(\varepsilon_y)\,\breve{}}(x)\text{ \ for all\ }x,y\in D,\]
see e.g.\ \cite[Theorem~4.1.6]{AG}, and we therefore obtain
\begin{align}
\label{to00}E_g(\mu_r^{\varepsilon,0})
  &=\int U_g^{\mu_r^{\varepsilon,0}}\,d\mu_r^{\varepsilon,0}=\int\Bigl(\,U_2^{\mu_r^{\varepsilon,0}}-
  U_2^{(\mu_r^{\varepsilon,0})\,\breve{}}\,\Bigr)\,d\mu_r^{\varepsilon,0}\\
{}&=\int U_2^{\mu_r}\,d\mu_r-\int U_2^{\mu_r}(-2\varepsilon,x_2,x_3)\,d\mu_r(x_1,x_2,x_3)\to0\notag
\end{align}
as $\varepsilon\downarrow0$, noting that
$U_2^{\mu_r}(-2\varepsilon,x_2,x_3)\to U_2^{\mu_r}(0,x_2,x_3)$
uniformly with respect to $(0,x_2,x_3)\in K_r$ as
$\varepsilon\downarrow0$.

Consider decreasing sequences $\{c_k\}_{k\in\mathbb N}$ and $\{r_k\}_{k\in\mathbb N}$ of the numbers
$c_k=2^{-k}$ and $r_k=2^{-2k}$. Then $c_k^2/r_k=1$, hence
\begin{equation}\label{exx}\sum_{k\in\mathbb N}\,c_k=1\text{ \  and \ }
\sum_{k\in\mathbb N}\,c_k^2/r_k=\infty.\end{equation} For
$k\in\mathbb N$ choose $0<\varepsilon_k<1$ small enough so that
\begin{equation}\label{small}\bigl\|\mu_{r_k}^{\varepsilon_k,k}\bigr\|_g=
\bigl\|\mu_{r_k}^{\varepsilon_k,0}\bigr\|_g<1,\end{equation}
which is possible in view of (\ref{to00}). Now define the functional
\[\mu(\varphi):=\sum_{k\in\mathbb N}\,c_k\mu_{r_k}^{\varepsilon_k,k}(\varphi)
\text{ \ for all\ }\varphi\in C_0(D).\] Since any compact subset of
$D$ has points in common with only finitely many (disjoint) disks
$K_{r_k}^{\varepsilon_k,k}$, $\mu$ thus defined is a positive Radon
measure on $D$ with $\mu(D)=1$, see the former equality in
(\ref{exx}). Furthermore, the partial sums
\[\eta_\ell:=\sum_{k=1}^{\ell}\,c_k\mu_{r_k}^{\varepsilon_k,k},\text{ \ where\ }\ell\in\mathbb N,\]
belong to $\mathcal E^+_g(D)$ with $\|\eta_\ell\|_g<1$, the latter
being clear from (\ref{small}) and the former equality in
(\ref{exx}) in view of the triangle inequality in $\mathcal E_g(D)$.
Since $\eta_\ell\to\mu$ vaguely in $\mathfrak M^+(D)$, hence
$\eta_\ell\otimes\eta_\ell\to\mu\otimes\mu$ vaguely in $\mathfrak
M^+(D\times D)$ \cite[Chapter~III, Section~5, Exercise~5]{B2}, we
obtain $\|\mu\|_g\leqslant 1$ from Lemma~\ref{lemma-semi} with
$X=D\times D$ and $\psi=g$.

On the other hand, being bounded, $\mu$ is extendible to a positive Radon measure on $\mathbb R^3$ and
\[E_2(\mu)\geqslant\sum_{k\in\mathbb N}\,E_2\bigl(c_k\mu_{r_k}^{\varepsilon_k,k}\bigr)=
\sum_{k\in\mathbb N}\,c_k^2E_2(\mu_{r_k})=\sum_{k\in\mathbb
N}\,c_k^2r_k^{-1}E_2(\mu_1)=\infty,\] where the last equality
follows from the latter equality in (\ref{exx}). This verifies
Example~\ref{counterex} for a half-space.

For treating the ball, apply the inversion relative to the sphere
with center $(2,0,0)$ and radius $2$. It maps the above half-space
$D$ on the ball $D^*$ centered at $(1,0,0)$ and with radius $1$. The
above measure $\mu$ has bounded Newtonian potential $U_2^\mu$ at the
point $(2,0,0)$
 because $\mu$ is bounded and supported by the closed strip $\{0\leqslant x_1\leqslant1\}$
 not containing $(2,0,0)$. Therefore, the Kelvin transform $\mu^*$ of $\mu$ is a bounded measure,
 see \cite[Eq.~(4.5.3)]{L}, and can be written in the form
\[\mu^*=\sum_{k\in\mathbb N}\,c_k\bigl(\mu_{r_k}^{\varepsilon_k,k}\bigr)^*,\]
the Kelvin transformation of positive measures being clearly
countably additive. Since $\kappa_2$-energy is preserved by Kelvin
transformation, so is $g_D^2$-energy of the measure
$\mu_{r_k}^{\varepsilon_k,k}\in\mathcal E_2^+(D)$, as seen by
combining \cite[Eqs.~(4.5.2), (4.5.4)]{L} and (\ref{eq1-2-hen})
above. Denoting by $g^*$ the Green kernel for the above ball $D^*$
we therefore obtain by (\ref{small})
\[\|\mu^*\|_{g^*}\leqslant
\sum_{k\in\mathbb
N}\,c_k\bigl\|\bigl(\mu_{r_k}^{\varepsilon_k,k}\bigr)^*\bigr\|_{g^*}=
\sum_{k\in\mathbb
N}\,c_k\bigl\|\mu_{r_k}^{\varepsilon_k,k}\bigr\|_g\leqslant 1.\] And
clearly $E_2(\mu^*)=E_2(\mu)=\infty$. This verifies
Example~\ref{counterex} also for a ball.
\end{example}

{\bf Acknowledgement.} The authors express their sincere gratitude
to the anonymous referees for valuable suggestions, helping us in
improving the exposition of the paper.


\begin{thebibliography}{99}
\bibitem{AG}D.H.~Armitage, S.J.~Gardiner, {\it Classical Potential Theory\/}, Springer, Berlin, 2001.

\bibitem{BC} B.~Beckermann, A.~Gryson, {\it Extremal rational functions on symmetric discrete
sets and superlinear convergence of the ADI method\/}, Constr.\ Approx. {\bf 32} (2010), 393--428.

\bibitem{BH} J.~Bliednter, W.~Hansen, {\it Potential Theory: An Analytic and Probabilistic Approach
to Balayage\/}, Springer, Berlin, 1986.

\bibitem{B1}
N.~Bourbaki, {\it Elements of Mathematics. General Topology. Chapters\/}~1--4,
Springer, Berlin, 1989.

\bibitem{B2} N.~Bourbaki, {\it Elements of Mathematics. Integration. Chapters\/}~1--6,
Springer, Berlin, 2004.

\bibitem{Brelo2} M.~Brelot, {\it On Topologies and Boundaries in
Potential Theory\/}, Lecture Notes in Math., vol.~175. Springer,
Berlin, 1971.

\bibitem{Car} H.~Cartan, {\it Th\'eorie du potentiel newtonien:
\'energie, capacit\'e, suites de potentiels\/}, Bull.\ Soc.\ Math.\ France
{\bf 73} (1945), 74--106.

\bibitem{D1} J.~Deny, {\it Les potentiels d'\'energie finie\/}, Acta Math.\ {\bf 82} (1950), 107--183.

\bibitem{D2} J.~Deny, {\it Sur la d\'{e}finition de
l'\'{e}nergie en th\'{e}orie du potentiel\/}, Ann.\ Inst.\ Fourier
Grenoble {\bf 2} (1950), 83--99.

\bibitem{DFHSZ} P.D.~Dragnev, B.~Fuglede, D.P.~Hardin, E.B.~Saff, N.~Zorii, {\it Minimum Riesz energy
problems for a condenser with touching plates\/}, Potential Anal.\ {\bf 44}  (2016), 543--577.

\bibitem{DFHSZ2} P.D.~Dragnev, B.~Fuglede, D.P.~Hardin, E.B.~Saff, N.~Zorii, {\it Constrained minimum
Riesz energy problems for a condenser with intersecting plates\/}, J.~Anal.\ Math., to appear.
ArXiv:1710.01950 (2017).

\bibitem{E2} R.~Edwards, {\it Functional Analysis. Theory and
Applications\/}, Holt, Rinehart and Winston, New York, 1965.

\bibitem{F1} B.~Fuglede, {\it On the theory of potentials in
locally compact spaces\/}, Acta Math.\ {\bf 103} (1960), 139--215.

\bibitem{Fu4} B.~Fuglede, {\it Capacity as a sublinear functional generalizing an integral\/},
Mat.\ Fys.\ Medd.\ Dan.\ Vid.\ Selsk.\ {\bf 38}, no.~7 (1971).

\bibitem{Fu5} B.~Fuglede, {\it Symmetric function kernels and sweeping of measures\/},
Analysis Math.\ {\bf 42} (2016), 225--259.

\bibitem{FZ} B.~Fuglede, N.~Zorii, {\it Green kernels associated with Riesz kernels\/},
Ann.\ Acad.\ Sci.\ Fenn.\ Math.\ {\bf 43} (2018), 121--145.

\bibitem{FZ-Pot} B.~Fuglede, N.~Zorii, {\it An alternative concept of Riesz energy of measures with
application to generalized condensers\/}, Potential Anal., https://doi.org/10.1007/s11118-018-9709-3

\bibitem{HWZ}
H.~Harbrecht, W.L.~Wendland, N.~Zorii, {\it Riesz minimal
energy problems on $C^{k-1,k}$-manifolds\/},  Math.\ Nachr.\ \textbf{287} (2014), 48--69.

\bibitem{L}
N.S.~Landkof, {\it Foundations of Modern Potential Theory\/}, Springer,
Berlin, 1972.

\bibitem{OWZ}
G.~Of, W.L.~Wendland,  N.~Zorii, {\it On the numerical
solution of minimal energy problems\/}, Complex Var.\
Elliptic Equ.\ \textbf{55} (2010), 991--1012.

\bibitem{O}
M.~Ohtsuka, {\it On potentials in locally compact spaces\/},
J.~Sci.\ Hiroshima Univ.\ Ser.~A-1 {\bf 25} (1961), 135--352.

\bibitem{ZPot1} N.~Zorii, {\it Interior capacities of condensers in
locally compact spaces\/}, Potential Anal.\ {\bf 35} (2011), 103--143.

\bibitem{Z9}
N.~Zorii, {\it Constrained energy problems with external fields for vector measures\/},
 Math.\ Nachr.\ {\bf 285} (2012), 1144--1165.

\bibitem{ZPot2} N.~Zorii, {\it Equilibrium problems for infinite dimensional vector potentials with
external fields\/}, Potential Anal.\ {\bf 38} (2013), 397--432.

\bibitem{ZPot3} N.~Zorii, {\it Necessary and sufficient conditions for the solvability of the Gauss
variational problem for infinite dimensional vector measures\/}, Potential Anal.\
{\bf 41} (2014), 81--115.

\end{thebibliography}
\end{document}